\documentclass[article]{amsart}
\usepackage{amssymb,amsfonts,amsmath,amsthm}
\usepackage[all,arc]{xy}
\usepackage{enumerate}
\usepackage{mathrsfs}
\usepackage[toc,page]{appendix}
\usepackage[left=3cm, right=3cm, bottom=3cm]{geometry}
\usepackage{graphicx}
\usepackage{tabularx}
\usepackage{multirow} % multirow in table
\usepackage{caption} % table caption with no number
\usepackage{url}
\usepackage{color}
\usepackage{tikz-cd}  % graph
\usepackage{mathdots} %\ddots, \vdots, \iddots
\usepackage{romannum} %Roman number
\usepackage{hyperref} %hyperlink for reference
\usepackage{float} %table follow paragraphp

\newtheorem{thm}{Theorem}[section]

\newtheorem*{thm*}{Theorem}
\newtheorem*{cor*}{Corollary}
\newtheorem*{prop*}{Proposition}

\newtheorem{prop}[thm]{Proposition}
\newtheorem{lem}[thm]{Lemma}

\theoremstyle{definition}
\newtheorem{defn}[thm]{Definition}

\newtheorem{con}[thm]{Construction}
\newtheorem*{conv*}{Convention}
\newtheorem{exmp}[thm]{Example}

\newtheorem*{notn*}{Notation}

\theoremstyle{remark}
\newtheorem{rem}[thm]{Remark}

\newtheorem*{idea*}{Idea}

\makeatletter
\let\c@equation\c@thm
\makeatother
\numberwithin{thm}{section}
\numberwithin{equation}{section}

\title[Moduli Spaces of Filtered G-local Systems on Curves]{Moduli Spaces of Filtered G-local Systems on Curves}

\author{Pengfei Huang and Hao Sun}

\begin{document}
\pagenumbering{arabic}
\maketitle
\begin{abstract}
In this paper, we construct the moduli spaces of filtered $G$-local systems on curves for an arbitrary reductive group $G$ over an algebraically closed field of characteristic zero. This provides an algebraic construction for the Betti moduli spaces in the tame nonabelian Hodge correspondence for vector bundles/principal bundles on noncompact curves. As a direct application, the tame nonabelian Hodge correspondence on noncompact curves holds not only for the relevant categories, but also for the moduli spaces.
\end{abstract}
	
\flushbottom
	
%\tableofcontents
%\newpage
	
\renewcommand{\thefootnote}{\fnsymbol{footnote}}
	
\section{Introduction}

\subsection{Background}
The celebrated nonabelian Hodge correspondence (NAHC) on a smooth projective variety over $\mathbb{C}$ establishes a connection among three types of objects on the variety: Higgs bundles, $D_X$-modules, and local systems \cite{Sim92,Sim94a,Sim94b}. This correspondence holds not only for the relevant categories, but also for the moduli spaces. For the noncompact case, by introducing extra structures that capture the asymptotic behaviour of local sections near the divisors, these three types of objects were generalized by Simpson into: filtered regular Higgs bundles, filtered regular $D_X$-modules and filtered local systems respectively, and thus the tame NAHC for vector bundles on noncompact curves was established with the help of tame harmonic bundles \cite{Sim90}. This correspondence is only given as an equivalence of the corresponding categories due to the lack of constructions of the corresponding moduli spaces. By now, the moduli spaces of filtered Higgs bundles (Dolbeault moduli spaces) and filtered $D_X$-modules (de Rham moduli spaces) have been constructed in numerous references (see \cite{Alf17,BY96,Ina13,Kon93,Sun20,Yok93}, for instance). However, a construction of the moduli spaces of filtered local systems (Betti moduli spaces) is still missing (see \cite[\S 1]{Sim97}). Moreover, since the stability of filtered local systems is not equivalent to the irreducibility of the corresponding representations of the fundamental group of the noncompact curve, the results in \cite{Rich88} cannot be applied directly to this situation. Therefore, we need a new approach to construct the Betti moduli spaces in the noncompact case. On the other hand, when dealing with principal bundles, Simpson's filtered structures (also called parabolic structures) are not enough to produce a satisfactory complete correspondence. This obstacle was overcome recently by the authors of \cite{HKSZ22} via parahoric torsors under the language of parahoric Bruhat--Tits group schemes, where a complete tame NAHC for principal bundles on noncompact curves was established for categories and claimed to be true for moduli spaces because they did not construct the Betti moduli space (moduli space of filtered $G$-local systems) directly.

The goal of this paper is to construct the Betti moduli spaces involved in the tame NAHC for vector/principal bundles on noncompact curves. As a direct application, the tame NAHC on noncompact curves holds for moduli spaces \cite{Sim90,HKSZ22}.

\subsection{Main Results}

In this paper, we always assume that $k$ is an algebraically closed field of characteristic zero, and $X$ is a connected smooth projective algebraic curve with genus $g$ over $k$. Let $\boldsymbol{D}$ be a set of distinct points in $X$. Denote by $X_{\boldsymbol{D}}:= X \backslash \boldsymbol{D}$ the noncompact curve.

In \S\ref{sect_fls}, we construct the moduli space of degree zero filtered local systems. It is well-known that representations of the fundamental group $\pi_1(X)$ can be regarded as representations of some special quiver. Therefore, local systems on $X$ are related to quiver representations (see \S\ref{subsect_rep_fund_grp}). For the noncompact case, we are considering filtered local systems on $X_{\boldsymbol{D}}$, which are local systems on $X_{\boldsymbol{D}}$ equipped with parabolic structures for each $x \in \boldsymbol{D}$ (Definition \ref{defn_fls}). Roughly speaking, a parabolic structure is a weighted filtration, and the type of filtrations is a parabolic subgroup regarded as the automorphism group of the filtration and the weights are rational numbers. Denote by $\boldsymbol{P}$ (resp. $\boldsymbol\theta$) a set of parabolic subgroups (resp. weights) indexed by $\boldsymbol{D}$. Although there are many methods to construct a quiver such that the quiver representations correspond to fundamental group representations, the key point is that the correspondence should also preserve the stability conditions of Definition \ref{defn_stab_num} and Definition \ref{defn_fls_stab}. In Construction \ref{cons_new_quiv}, we construct a quiver $\mathscr{Q}^{\boldsymbol{D}}$ such that the orbits in the (quiver) representation space $\mathscr{R}(\mathscr{Q}^{\boldsymbol{D}},\mathscr{I}^{\boldsymbol{D}},[\boldsymbol{P}])$ are in one-to-one correspondence with the isomorphism classes of filtered local systems of type $[\boldsymbol{P}]$ (Proposition \ref{prop_para_and_orbit}). Inspired by King's work \cite{King94}, with a good choice of the character $\chi_{\boldsymbol\theta}$ (Construction \ref{cons_char_theta}), we prove that the stability of filtered local systems of type $[\boldsymbol{P}]$ is equivalent to the $\chi_{\boldsymbol\theta}$-stability of the corresponding point in $\mathscr{R}(\mathscr{Q}^{\boldsymbol{D}},\mathscr{I}^{\boldsymbol{D}},[\boldsymbol{P}])$ (Proposition \ref{prop_stab_local_sys_and_quiv_rep}). Finally, applying the geometric invariant theory, we construct the moduli space of degree zero filtered local systems.
\begin{thm}[Theorem \ref{thm_tame_Betti_lc}]\label{thm_tame_Betti_lc_intro}
There exists a quasi-projective variety $\mathcal{M}_{\rm B}(X_{\boldsymbol{D}},[\boldsymbol{P}],\boldsymbol\theta)$ as the moduli space of degree zero semistable filtered local systems of type $[\boldsymbol{P}]$ with weights $\boldsymbol\theta$. There exists an open subset $\mathcal{M}^s_{\rm B}(X_{\boldsymbol{D}},[\boldsymbol{P}],\boldsymbol\theta)$, whose points correspond to isomorphism classes of degree zero stable filtered local systems of type $[\boldsymbol{P}]$ with weights $\boldsymbol\theta$.
\end{thm}
\noindent We refer the reader to Theorem \ref{thm_tame_Betti_lc} for a precise statement. Since the stability condition of filtered local systems we consider (Definition \ref{defn_fls_stab}) is the one given by Simpson \cite{Sim90}, the moduli space $\mathcal{M}_{\rm B}(X_{\boldsymbol{D}}, [\boldsymbol{P}],\boldsymbol\theta)$ is the Betti moduli space in the tame NAHC for vector bundles on noncompact curves, where the subscript $``{\rm B}"$ is for ``Betti". 

In \S\ref{sect_fgls}, we use a similar approach as we did in \S\ref{sect_fls} to construct the moduli space of degree zero filtered $G$-local systems. The main result is given as follows:
\begin{thm}[Theorem \ref{thm_tame_Betti_lc_G}]
There exists a quasi-projective variety $\mathcal{M}_{\rm B}(X_{\boldsymbol{D}},G,\boldsymbol\theta)$ as the moduli space of degree zero $R$-semistable $\boldsymbol\theta$-filtered $G$-local systems. There exists an open subset $\mathcal{M}^s_{\rm B}(X_{\boldsymbol{D}},G,\boldsymbol\theta)$, whose points correspond to isomorphism classes of degree zero $R$-stable $\boldsymbol\theta$-filtered $G$-local systems.
\end{thm}
\noindent The stability condition for filtered $G$-local systems (Definition \ref{defn_stab_cond_G}) follows from that of principal bundles given by Ramanathan \cite{Ram75,Ram96a,Ram96b}, which is called the \emph{$R$-stability condition} in this paper. When $G={\rm GL}_n(k)$, the $R$-stability condition of filtered local systems is equivalent to the usual slope stability (Definition \ref{defn_fls_stab}). Moreover, the $R$-stability condition is the one the authors used to establish the tame NAHC for principal bundles on noncompact curves \cite{HKSZ22}.

\subsection{Applications and Remarks}

Yamakawa constructed the moduli space of filtered local systems on $\mathbb{P}^1$ when $G={\rm GL}_n(\mathbb{C})$ \cite{Yam08}. Compared to his work, Theorem \ref{thm_tame_Betti_lc_intro} is more general, which gives the existence of the moduli space of degree zero filtered local systems on arbitrary algebraic curves. By a similar technique as Simpson did in \cite{Sim94a,Sim94b}, one can see that the tame NAHC on noncompact curves given in \cite[Theorem, p. 718]{Sim90} also holds at the level of moduli spaces. The same argument holds for the Betti moduli space in the tame NAHC for principal bundles on noncompact curves \cite{HKSZ22} with respect to Theorem \ref{thm_tame_Betti_lc_G}. Since the result for principal bundles contains Simpson's result as a special case, we only give the construction of the Betti moduli spaces for principal bundles appearing in the statement of the tame NAHC for principal bundles in \S\ref{sect_Betti_moduli} (see Theorem \ref{thm_HKSZ1.1} or \cite[Theorem 1.1]{HKSZ22}). 

When we consider the Betti moduli spaces in the tame NAHC on noncompact curves, we fix the data of monodromies/residues at punctures in order to get a more precise description of the one-to-one correspondence. Let $\boldsymbol\gamma$ be a collection of weights indexed by $\boldsymbol{D}$. Fixing a collection of the data of monodromies $M_{\boldsymbol\gamma}$, we construct the Betti moduli space $\mathcal{M}_{\rm B}(X_{\boldsymbol{D}},G,\boldsymbol\gamma, M_{\boldsymbol\gamma})$, whose points correspond to $S$-equivalence classes of degree zero $R$-semistable $\boldsymbol\gamma$-filtered $G$-local systems with given monodromy data $M_{\boldsymbol\gamma}$ (Theorem \ref{thm_Betti_moduli}). 

In this paper, the weights of filtered local systems are always assumed to be rational because we follow King's approach to give the construction of moduli spaces. If we take real weights, we cannot associate it with a well-defined character (see Construction \ref{cons_char_theta} and \S\ref{subsect_stab_R}). However, given a collection of real weights $\boldsymbol\theta$, we can always find a collection of ``nearby" rational weights $\boldsymbol\theta'$ such that the parabolic subgroups are the same $P_{\theta_x} = P_{\theta'_x}$ and the stability conditions for stable $\boldsymbol\theta'$-filtered and stable $\boldsymbol\theta$-filtered $G$-local systems are equivalent. Thus, the stable locus of the moduli space of $\boldsymbol\theta'$-filtered $G$-local systems can be regarded as that of $\boldsymbol\theta$-filtered $G$-local systems (see Remark \ref{rem_real_weight}). This approach of dealing with real weights is also considered in \cite{BBP17} and \cite[\S8]{IS08}.

For the wild case, the authors established the NAHC for principal objects in the unramified case \cite{HS22} and the approach given in this paper can be used to construct the wild Betti moduli space, which will be discussed in a future project for more details.

\vspace{2mm}

\textbf{Acknowledgments}.
The authors would like to thank Andres Fernandez Herrero, Georgios Kydonakis, Carlos Simpson, Tao Su, Szil\'ard Szab\'o, Markus Reineke, and Lutian Zhao for helpful discussions and their interest in this work, especially Georgios Kydonakis for his very helpful comments and suggestions on a very early version of this paper. The authors would like to thank anonymous referees for pointing out a mistake in an early version of this paper and giving insightful comments and constructive suggestions. 
P. Huang would like to express deep gratitude to the Institut des Hautes \'Etudes Scientifiques, and the Max Planck Institute for Mathematics in the Sciences for their kind hospitality and support during the production of this paper.

This project is funded by National Key R$\&$D Program of China (No. 2022YFA1006600), 
the Deutsche Forschungsgemeinschaft (DFG, Projektnummer 547382045), and Guangdong Basic and Applied Basic Research Foundation (2024A1515011583). 

\vspace{2mm}

\section{Preliminaries}\label{sect_prel}
We briefly review the numerical criterion for quiver representations given in \cite{King94} and some properties of quiver representations, of which good references are \cite{ASS06,Rei08}. Finally, in \S\ref{subsect_rep_fund_grp}, we construct a quiver such that fundamental group representations can be regarded as representations of this quiver. Thus, local systems are related to quiver representations.

\subsection{Numerical Criterion}\label{subsect_num_crit}
In this subsection, we review the numerical criterion given in \cite[\S 2]{King94}, which is an application of the famous Mumford's numerical criterion in geometric invariant theory (GIT) \cite[Chapter 2]{MFK94}. Let $G$ be a connected reductive group over $k$. Let $\mathscr{R}$ be a smooth affine variety over $k$ together with a $G$-action. Denote by $\Delta$ the kernel of this action. Let $L$ be the trivial line bundle on $\mathscr{R}$. A $G$-linearization of $L$ is given by a character $\chi$ of $G$ in this case. Moreover, as a trivial line bundle on $\mathscr{R}$, a function in $k[\mathscr{R}]$ is regarded as a section of $L$. Given a character $\chi$ of $G$, a function $f(x) \in k[\mathscr{R}]$ is a relative invariant of weight $\chi^n$, if
\begin{align*}
    f(g \cdot x)=\chi^n(g) f(x),
\end{align*}
where $n$ is a non-negative integer. Denote by $k[\mathscr{R}]^{G,\chi^i}$ the subset of relative invariants of weight $\chi^i$. If $f \in k[\mathscr{R}]^{G,\chi^i}$ and $g \in k[\mathscr{R}]^{G,\chi^j}$, then $f \cdot g \in k[\mathscr{R}]^{G,\chi^{i+j}}$. Thus, $\bigoplus_{i \geq 0} k[\mathscr{R}]^{G,\chi^i}$ has a natural graded structure. Denote by
\begin{align*}
\mathscr{R} /\!\!/ (G,\chi) := {\rm Proj}\Big(\bigoplus_{i \geq 0}  k[\mathscr{R}]^{G,\chi^i}\Big).
\end{align*}
The quotient $\mathscr{R} /\!\!/ (G,\chi)$ is known as the GIT quotient, and it is a scheme projective over $k[\mathscr{R}]^{G,\chi^0}=k[\mathscr{R}]^{G}$. If $k[\mathscr{R}]^{G} \cong k$, then $\mathscr{R} /\!\!/ (G,\chi)$ is a projective variety. Now we consider the stability condition introduced by Mumford:
\begin{defn}\label{defn_stab_num}
A point $x \in \mathscr{R}$ is \emph{$\chi$-semistable}, if there is a relative invariant $f \in k[\mathscr{R}]^{G,\chi^n}$ for $n \geq 1$ such that $f(x)\neq 0$. It is \emph{$\chi$-stable}, if it is $\chi$-semistable, $\dim(G \cdot x) = \dim G / \Delta$, and the $G$-action on the set $\{x \in \mathscr{R} \text{ }|\text{ } f(x) \neq 0\}$ is closed.
\end{defn}

There is a natural pairing
\begin{align*}
    \langle \cdot,\cdot \rangle: {\rm Hom}(\mathbb{G}_m,G) \times {\rm Hom}(G, \mathbb{G}_m) \rightarrow \mathbb{Z}
\end{align*}
of cocharacters and characters such that $\langle \mu,\chi \rangle = n$ if $\chi(\mu(t))=t^n$. Note that this pairing can be defined over rational coefficients. As an application of Mumford's numerical criterion, we have the following result to determine the stability:

\begin{prop}[Proposition 2.5 in \cite{King94}]\label{prop_king2.5}
A point $x \in \mathscr{R}$ is $\chi$-semistable if and only if $\chi(\Delta)=\{1\}$ and given any cocharacter $\mu$ of $G$, if $\lim_{t\rightarrow 0} \mu(t) \cdot x$ exists, then $\langle \mu, \chi \rangle \geq 0$. A point $x \in \mathscr{R}$ is $\chi$-stable if and only if the cocharacters $\mu$, for which $\lim_{t\rightarrow 0} \mu(t) \cdot x$ exists and $\langle \mu,\chi \rangle=0$, are with values in $\Delta$.
\end{prop}

Denote by $\mathscr{R}^{(s)s}_\chi \subseteq \mathscr{R}$ the open subset of $\chi$-(semi)stable points. The points in $\mathscr{R} /\!\!/ (G,\chi)$ are in one-to-one correspondence with equivalence classes of points in $\mathscr{R}^{ss}_\chi$, where the equivalence of points in $\mathscr{R}^{ss}_\chi$ is given as follows: $x \sim y$ if $\overline{G \cdot x}\ \cap\ \overline{G \cdot y}\ \cap\ \mathscr{R}^{ss}_\chi \neq \emptyset$. This equivalence is called the \emph{GIT equivalence}.

\begin{prop}[Proposition 2.6 in \cite{King94}]\label{prop_king2.6}
Let $x, y \in \mathscr{R}^{ss}_{\chi}$ be two points. An orbit $G \cdot x$ is closed in $\mathscr{R}^{ss}_\chi$ if and only if when the limit $\lim_{t\rightarrow 0} \mu(t) \cdot x$ exists for any cocharacter $\mu$ of $G$ satisfying $\langle \mu,\chi \rangle = 0$, the limit $\lim_{t\rightarrow 0} \mu(t) \cdot x$ is in the orbit $G \cdot x$. Furthermore, $x \sim y$ if and only if there are cocharacters $\mu_1$ and $\mu_2$ such that $\langle \mu_1,\chi \rangle = \langle \mu_2,\chi \rangle = 0$ and the limits $\lim_{t \rightarrow 0} \mu_1(t) \cdot x$ and $\lim_{t \rightarrow 0} \mu_2(t) \cdot y$ are in the same closed $G$-orbit.
\end{prop}

Propositions \ref{prop_king2.5} and \ref{prop_king2.6} will be used later to construct the moduli spaces of filtered $G$-local systems.

\subsection{Representations of Quivers}\label{subsect_rep_quiv}
A \emph{quiver} $Q$ is a directed graph $(Q_0,Q_1,s,t)$, where $Q_0$ is the set of vertices, $Q_1$ is the set of arrows and $s,t: Q_1 \rightarrow Q_0$ are the source map and target map, respectively. We use the notation $Q=(Q_0,Q_1)$ for simplicity. A quiver $Q$ is \emph{finite} if $Q_0$ and $Q_1$ are both finite sets, and $Q$ is \emph{connected} if it is connected as a graph. In this paper, quivers we consider are always supposed to be finite and connected. A \emph{path of length $\ell$} is a sequence of arrows $w = (a_1,\dots,a_{\ell})$ such that $s(a_{i+1}) = t(a_i)$ for $1 \leq i \leq \ell -1$, i.e.
\begin{align*}
    \bullet \xrightarrow{a_1} \bullet \xrightarrow{a_2} \dots \xrightarrow{a_{\ell}} \bullet.
\end{align*}
We also use the notation $w=a_{\ell} \dots a_{1}$ for convenience, which is regarded as compositions of arrows. The source (resp. target) of $w$ is defined to be the source (resp. target) of $a_1$ (resp. $a_{\ell}$). Moreover, a path $w=a_{\ell} \dots a_{1}$ is a \emph{loop} if $s(a_1) = t(a_{\ell})$. We associate each point $v \in Q_0$ with a path of length zero, with a slight abuse of notation, denote by $Q_0$ the set of all paths of length zero. Then, let $Q_{\ell}$ be the set of all paths of length $\ell$. Denote by $kQ$ the \emph{path algebra} of $Q$, of which the basis of the underlying $k$-vector space is the set of all paths of length $\ell \geq 0$ in $Q$ and the product is given by the composition of paths. Since
\begin{align*}
kQ = \bigoplus_{\ell \geq 0} k Q_\ell, \quad k Q_i \cdot k Q_j \subseteq kQ_{i+j},
\end{align*}
the path algebra has a natural graded structure.

Let $Q$ be a (finite and connected) quiver. A \emph{representation} of $Q$ is denoted by $M=(M_v,\phi_a)_{v \in Q_0, a \in Q_1}$, where $M_v$ is a $k$-vector space for $v \in Q_0$ and $\phi_a$ is a $k$-linear map $\phi_a: M_{s(a)} \rightarrow M_{t(a)}$ for $a \in Q_1$. Let $w=a_\ell \dots a_{1}$ be a path, and then $\phi_w = \phi_{a_\ell} \cdots \phi_{a_1}$. A representation $M$ of $Q$ is \emph{of finite dimension} if $M_v$ is of finite dimension for each $v \in Q_0$, and the \emph{dimension vector} of $M$ is given by $\boldsymbol{n}=\{n_v\}_{v \in Q_0}$, where $n_v = \dim M_v$. If the vector spaces $M_v$ are of the same dimension $n$, then the representation $M$ is said to be \emph{of dimension $n$}. Denote by ${\rm rep}(Q)$ the category of finite dimensional representations of $Q$. Note that this category is an abelian category and it is equivalent to the category of $kQ$-modules of finite dimension \cite{ASS06}. Denote by $K_0({\rm rep}(Q))$ the Grothendieck group of the abelian category ${\rm rep}(Q)$, and let $\Theta: K_0({\rm rep}(Q)) \rightarrow \mathbb{Z}$ be an additive function on this group. 

\begin{defn}\label{defn_stab_rep_quiv}
A finite dimensional representation $M$ of $Q$ is \emph{$\Theta$-semistable} (resp. \emph{$\Theta$-stable}) if $\Theta(M)=0$ and for any nontrivial proper submodule $M' \subseteq M$, we have $\Theta(M') \geq 0$ (resp. $\Theta(M') > 0$). Two $\Theta$-semistable representations are \emph{$S$-equivalent} if they have the same composition factors in the category of $\Theta$-semistable representations.
\end{defn}

Now we fix a dimension vector $\boldsymbol{n}$. The isomorphism classes of representations of $Q$ with dimension vector $\boldsymbol{n}$ are in one-to-one correspondence with orbits in
\begin{align*}
    \mathscr{R}(Q,\boldsymbol{n}) = \bigoplus_{a \in Q_1} {\rm Hom}(k^{n_{s(a)}}, k^{n_{t(a)}})
\end{align*}
under the action of
\begin{align*}
    {\rm GL}(\boldsymbol{n}) = \prod_{v \in Q_0} {\rm GL}_{n_v}(k),
\end{align*}
where the action is defined as
\begin{align*}
(g_{v})_{v \in Q_0} \cdot (\phi_a)_{a \in Q_1} := ( g_{t(a)} \cdot \phi_a \cdot g^{-1}_{s(a)})_{a \in Q_1},
\end{align*}
for $(g_{v})_{v \in Q_0} \in {\rm GL}(\boldsymbol{n})$ and $(\phi_a)_{a \in Q_1} \in \mathscr{R}(Q,\boldsymbol{n})$. Then, $\mathscr{R}(Q,\boldsymbol{n})$ is called the \emph{space of quiver representations}. Similarly, we define
\begin{align*}
    \mathscr{R}(Q,n):=\bigoplus_{a \in Q_1} {\rm Hom}(k^n,k^n),
\end{align*}
and the isomorphism classes of representations of $Q$ with dimension $n$ are in one-to-one correspondence with orbits in $\mathscr{R}(Q,n)$ under the action of ${\rm GL}(\boldsymbol{n})$.

Given an additive function $\Theta: K_0({\rm rep}(Q)) \rightarrow \mathbb{Z}$, we define a character of ${\rm GL}(\boldsymbol{n})$ as follows
\begin{align*}
    \chi_{\Theta}(\boldsymbol{g}):= \prod_{v \in Q_0} \det(g_v)^{\Theta_v},
\end{align*}
where $\boldsymbol{g}:=(g_v)_{v \in Q_0}$ and the integers $\Theta_v$ are given from the formula $\Theta(M) = \sum_{v \in Q_0} \Theta_v \dim M_v$ for any representation $M$ of $Q$.

\begin{thm}[Proposition 3.1 in \cite{King94}]\label{thm_chi_stab_and_stab}
A point in $\mathscr{R}(Q,\boldsymbol{n})$ is $\chi_\Theta$-semistable (resp. $\chi_\Theta$-stable) if and only if the corresponding representation of $Q$ is $\Theta$-semistable (resp. $\Theta$-stable).
\end{thm}

\subsection{Relations in Quivers}

Let $Q$ be a (finite and connected) quiver, and let $kQ$ be the path algebra. A \emph{relation} (or a \emph{$k$-relation}) in $Q$ is a $k$-linear combination of paths in $Q$ such that the sources and targets of all paths are the same. Let $I$ be a set of relations in $Q$, and if there is only one relation in $I$, we directly use $I$ for the single relation. A representation $M$ of $Q$ satisfies the relations in $I$ if $\phi_w = 0$ for any relation $w \in I$. We define
\begin{align*}
\mathscr{R}(Q,I,\boldsymbol{n}) = \{ (\phi_a)_{a \in Q_1} \text{ }|\text{ } \phi_w = 0 \text{ for } w \in I \}.
\end{align*}
If the subset $\mathscr{R}(Q,I,\boldsymbol{n}) \subseteq \mathscr{R}(Q,\boldsymbol{n})$ is invariant under ${\rm GL}(\boldsymbol{n})$, then the ${\rm GL}(\boldsymbol{n})$-orbits in $\mathscr{R}(Q,I,\boldsymbol{n}) \subseteq \mathscr{R}(Q,\boldsymbol{n})$ are in one-to-one correspondence with the isomorphism classes of representations of $Q$ with dimension vector $\boldsymbol{n}$ satisfying the relations in $I$. As a special case, given a single loop $w$ with source and target $v$, let $I = w - v$. We define
\begin{align*}
\mathscr{R}(Q,I,\boldsymbol{n}) := \{ (\phi_a)_{a \in Q_1} \text{ }|\text{ } \phi_I = 0 \}.
\end{align*}
Note that the relation $\phi_I=0$ is equivalent to the condition that $\phi_w$ is the identity matrix (at the vertex $v$). 

We give an example to end this subsection. This example is helpful for understanding the equivalence between filtered local systems and quiver representations in Proposition \ref{prop_para_and_orbit}.
\begin{exmp}\label{exmp_basic_exmp}
We define $Q=(Q_0,Q_1)$ to be the quiver as follows
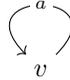
\begin{center}
\begin{tikzcd}
v \arrow[loop,"a" description]
\end{tikzcd}
\end{center}
which consists of a single vertex and a single loop. Given a positive integer $n$, the subset
\begin{align*}
    \mathscr{R}(Q,\det,n):=\{ \phi_a \in {\rm GL}_n(k) \} \subseteq \mathscr{R}(Q,n)
\end{align*}
consists of all homomorphisms $\phi_a$ with nonzero determinant. Clearly, $\mathscr{R}(Q,\det,n)$ is invariant under the action of ${\rm GL}_n(k)$.

Now let $Q'=(Q'_0,Q'_1)$ be the quiver
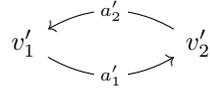
\begin{center}
\begin{tikzcd}
v'_1 \arrow[rr, bend right, "a'_1" description] & & v'_2 \arrow[ll, bend right, "a'_2" description]
\end{tikzcd}
\end{center}
which consists of two vertices and two arrows. Let $I'=a'_2 a'_1 - v'_1$ be a relation. Consider
\begin{align*}
\mathscr{R}(Q',I',n)=\{ (\phi_{a'_1}, \phi_{a'_2}) \text{ } | \text{ } \phi_{a'_2}\phi_{a'_1} = {\rm id}, \text{ } \phi_{a'_i} \in {\rm End}(k^n) \}.
\end{align*}
Note that the condition $\phi_{a'_2}\phi_{a'_1} = {\rm id}$ implies that $\phi_{a'_i} \in {\rm GL}_n(k)$. Moreover, the subset $\mathscr{R}(Q',I',n) \subseteq \mathscr{R}(Q',n)$ is invariant under the action of $\prod_{v \in Q'_0} G_v$, where $G_v = {\rm GL}_n(k)$.

It is clear that ${\rm GL}_n(k)$-orbits in $\mathscr{R}(Q,\det,n)$ are in one-to-one correspondence with $\left( \prod_{v \in Q'_0} G_v \right)$-orbits in $\mathscr{R}(Q',I',n)$. 
\end{exmp}

\subsection{Fundamental Group Representations and Local Systems}\label{subsect_rep_fund_grp}
Let $X$ be a connected smooth projective algebraic curve over $k$ with genus $g$. The fundamental group of $X$ at a base point $x_0$ is
\begin{align*}
    \pi_1(X,x_0)=\langle\ a_i,b_j, 1 \leq i,j \leq g \text{ }|\text{ } \prod_{i=1}^g [a_i,b_i] = {\rm id}\ \rangle.
\end{align*}
If there is no ambiguity, we omit the base point $x_0$ and use the notation $\pi_1(X)$ for the fundamental group. Fixing a positive integer $n$, we consider the space of fundamental group representations
\begin{align*}
{\rm Hom}(\pi_1(X),{\rm GL}_n(k)).
\end{align*}
We want to construct a quiver $\mathscr{Q}'$ such that fundamental group representations can be regarded as representations of $\mathscr{Q}'$, and one possible construction of such a quiver $\mathscr{Q}'=(\mathscr{Q}'_0,\mathscr{Q}'_1)$ is given as follows. The set of vertices $\mathscr{Q}'_0=\{v_0\}$ contains only one element, which can be regarded as the base point $x_0 \in X$. The set of arrows $\mathscr{Q}'_1$ contains $4g$ elements and denote them by
\begin{align*}
    a_i,b_i,c_i,d_i: v_0 \rightarrow v_0 \quad \text{ for } \quad 1 \leq i \leq g. 
\end{align*}
Since $\mathscr{Q}'_0$ has a single vertex, a positive integer $n$ is therefore regarded as a dimension vector for representations of $\mathscr{Q}'$. Let $\mathscr{I}' = \{ I'_0, I'_{ai}, I'_{bi}, 1 \leq i \leq g \}$ be a set of relations, where
\begin{align*}
    I'_0 = \prod_{i=1}^g (a_i b_i c_i d_i) - v_0, \quad I'_{ai} = a_i c_i - v_0, \quad I'_{bi} = b_i d_i - v_0.
\end{align*}
A representation $(\phi_a)_{a \in \mathscr{Q}'_1}$ of $\mathscr{Q}'$ satisfying relations in $\mathscr{I}$ means that
\begin{align*}
    \phi_{a_i} = \phi_{c_i}^{-1}, \quad \phi_{b_i} = \phi_{d_i}^{-1}, 
\end{align*}
and
\begin{align*}
    \prod_{i=1}^g [\phi_{a_i} , \phi_{b_i}]  = {\rm id}.
\end{align*}
Since the above relations imply that $\phi_a \in {\rm GL}_n(k)$ for each arrow $a \in \mathscr{Q}'_1$,  we have
\begin{align*}
\mathscr{R}(\mathscr{Q}',\mathscr{I}',n) \cong {\rm Hom}(\pi_1(X),{\rm GL}_n(k)).
\end{align*}
In this way, we construct a quiver $\mathscr{Q}'$ such that fundamental group representations are equivalent to quiver representations. We also want to remind the reader that the construction of such a quiver is not unique.

\begin{rem}
In the above construction, the arrow $c_i$ (resp. $d_i$) is regarded as the ``inverse" of $a_i$ (resp. $b_i$). We omit $c_i$ and $d_i$ to simplify notations and relations in the construction of the quiver $\mathscr{Q}'$. The simplified notation will be used in the rest of the paper.

Let $\mathscr{Q}_0 = \{v_0\}$ and $\mathscr{Q}_1$ contains $2g$ arrows $a_i,b_i$ for $1 \leq i \leq g$. Define a single relation \begin{align*}
    \mathscr{I} = \prod_{i=1}^g [a_i,b_i] - v_0. 
\end{align*}
Then, we define a subset $\mathscr{R}(\mathscr{Q},\mathscr{I},n)$ of $\mathscr{R}(\mathscr{Q},n)$ as
\begin{align*}
\mathscr{R}(\mathscr{Q},\mathscr{I},n)=\{(\phi_a)_{a \in \mathscr{Q}_1} \text{ }|\text{ } \phi_{\mathscr{I}}=0\}.
\end{align*}
Clearly, 
\begin{align*}
\mathscr{R}(\mathscr{Q},\mathscr{I},n) \cong {\rm Hom}(\pi_1(X),{\rm GL}_n(k)).
\end{align*}

Strictly speaking, the relation $\mathscr{I} = \prod_{i=1}^g [a_i,b_i] - v_0$ here is not a well-defined relation in $\mathscr{Q}$ because the ``inverse" arrows $a_i^{-1}$ and $b^{-1}_i$ are not defined. In the following, we still say that $\mathscr{I}$ is a relation of $\mathscr{Q}$ for convenience. We would like to remind the reader again that $(\mathscr{Q},\mathscr{I})$ is a simplified notation for $(\mathscr{Q}',\mathscr{I}')$, and the advantage of $(\mathscr{Q},\mathscr{I})$ is that it is closely related to the fundamental group. 
\end{rem} 

Fundamental group representations are also related to local systems. A \emph{local system} (of rank $n$) on $X$ is regarded as a pair $(\mathscr{L}_{x_0},\rho)$ in this paper, where $\mathscr{L}_{x_0}$ is a $k$-vector space of dimension $n$ and $\rho: \pi_1(X,x_0) \rightarrow {\rm GL}(\mathscr{L}_{x_0}) \cong {\rm GL}_n(k)$ is a representation. Denote by $[\rho]$ the corresponding point in ${\rm Hom}(\pi_1(X,x_0),{\rm GL}_n(k))$. We refer the reader to \cite{Sim94a,Sim94b} for a more detailed discussion about local systems, and moreover, the definition of local systems does not depend on the base point $x_0$. Two local systems $(\mathscr{L}_{x_0},\rho)$ and $(\mathscr{L}'_{x_0},\rho')$ are isomorphic if there exists an isomorphism $\phi: \mathscr{L}_{x_0} \rightarrow \mathscr{L}'_{x_0}$ of vector spaces such that $\phi \circ \rho = \rho' \circ \phi$. The isomorphism classes of local systems of rank $n$ are in one-to-one correspondence with ${\rm GL}_n(k)$-orbits in ${\rm Hom}(\pi_1(X),{\rm GL}_n(k))$. In conclusion, we have a (one-to-one) correspondence among fundamental group representations, quiver representations and local systems.
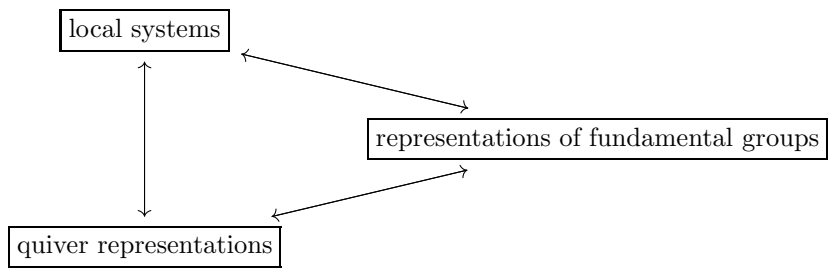
\begin{center}
\begin{tikzcd}
\fbox{local systems} \arrow[dd] \arrow[rd] & \\
& \fbox{representations of fundamental groups} \arrow[lu]  \arrow[ld] \\
\fbox{quiver representations} \arrow[uu] \arrow[ru] &
\end{tikzcd}
\end{center}

\section{Moduli Space of Filtered Local Systems}\label{sect_fls}

Let $X$ be a connected smooth projective algebraic curve over $k$ with genus $g$, and let $\boldsymbol{D}$ be a reduced effective divisor on $X$. The divisor $\boldsymbol{D}$ is actually a sum of finitely many distinct points, and therefore, it is also regarded as a set. Denote by $X_{\boldsymbol{D}}:=X \backslash \boldsymbol{D}$ the punctured curve. The fundamental group of $X_{\boldsymbol{D}}$ is
\begin{align*}
     \pi_1(X_{\boldsymbol{D}})=\langle\ a_i,b_j, c_x, 1 \leq i,j \leq g, x \in \boldsymbol{D} \text{ }|\text{ } \prod_{i=1}^g [a_i,b_i] \prod_{x \in \boldsymbol{D}} c_x = {\rm id}\ \rangle,
\end{align*}
where $c_x$ is the homotopy class of loops around the puncture $x \in \boldsymbol{D}$. As in \S\ref{subsect_rep_fund_grp}, a point $[\rho] \in {\rm Hom}(\pi_1(X_{\boldsymbol{D}}) , {\rm GL}_n(k))$ can be considered as a local system $\mathscr{L} = (\mathscr{L}_{x_0},\rho)$ of rank $n$ on $X_{\boldsymbol{D}}$.

The goal of this section is to construct the moduli space of filtered local systems on $X_{\boldsymbol{D}}$. ``Filtered" in this context refers to weighted filtrations, which are also called parabolic structures. When equipped with these additional structures, filtered local systems are the appropriate entities in the topological facet of the tame nonabelian Hodge correspondence. With the idea in \S\ref{subsect_rep_fund_grp}, we construct a quiver $\mathscr{Q}^{\boldsymbol{D}}$ and prove that there is a one-to-one correspondence between isomorphism classes of filtered local systems and orbits in the space of quiver representations (Proposition \ref{prop_para_and_orbit}). Furthermore, this correspondence also preserves the stability conditions (Proposition \ref{prop_stab_local_sys_and_quiv_rep}). Then, we give the construction of the moduli space of filtered local systems (Theorem \ref{thm_tame_Betti_lc}).

\subsection{Quasi-parabolic Local Systems}\label{subsect_qplc}

Let $V$ be a vector space of dimension $n$. A \emph{quasi-parabolic structure} on $V$ is a filtration of subspaces
\begin{align*}
V = V_1 \supsetneq V_2 \supsetneq \dots \supsetneq V_{n'+1}  = \{0\},
\end{align*}
where $n'\leq n$ is a positive integer, and denote by $V_{\bullet}$ a quasi-parabolic structure on $V$. Furthermore, define
\begin{align*}
    \dim (V_{i} / V_{i+1}) = \lambda_i,
\end{align*}
and $\lambda=(\lambda_1,\dots,\lambda_{n'})$ is a partition of $n$. Given a parabolic subgroup $P \subseteq {\rm GL}(V)$, if $P = {\rm Aut}(V_\bullet)$ (resp. $P$ is conjugate to ${\rm Aut}(V_\bullet)$), then we say that $V_\bullet$ is of \emph{type $P$} (resp. \emph{type $[P]$}). 

Now we consider local systems on $X_{\boldsymbol{D}}$. Since local systems do not depend on the choice of base points, let $(\mathscr{L},\rho)$ be a local system of rank $n$ on $X_{\boldsymbol{D}}$, where $\mathscr{L}$ is a vector space of dimension $n$ and $\rho: \pi_1(X_{\boldsymbol{D}}) \rightarrow {\rm GL}(\mathscr{L})$ is a representation.

\begin{defn}\label{defn_quasi_parab_lc}
A \emph{quasi-parabolic local system $\mathscr{L}_{\bullet}$} of rank $n$ on $X_{\boldsymbol{D}}$ is a local system $(\mathscr{L},\rho)$ of rank $n$ on $X_{\boldsymbol{D}}$ equipped with a quasi-parabolic structure $\mathscr{L}_{x,\bullet}$ for each puncture $x \in \boldsymbol{D}$ compatible with $\rho$. More precisely, denote by
\begin{align*}
\mathscr{L}_{x,\bullet}: \mathscr{L} = \mathscr{L}_{x,1} \supsetneq \dots \supsetneq \mathscr{L}_{x,n_x+1} = \{ 0\}
\end{align*}
the quasi-parabolic structure on $\mathscr{L}$ at the puncture $x \in \boldsymbol{D}$, where $n_x$ is a positive integer, and we have $(\rho(c_x))(\mathscr{L}_{x,i}) \subseteq \mathscr{L}_{x,i}$. Since $\rho(c_x) \in {\rm GL}(\mathscr{L})$, we indeed have $(\rho(c_x))(\mathscr{L}_{x,i}) =\mathscr{L}_{x,i}$.

Given a collection $\boldsymbol{P}:=\{P_{x}, x \in \boldsymbol{D}\}$ of parabolic subgroups in ${\rm GL}(\mathscr{L})$, a quasi-parabolic local system $\mathscr{L}_{\bullet}$ is of \emph{type $\boldsymbol{P}$} (resp. \emph{type $[\boldsymbol{P}]$}) if the quasi-parabolic structure $\mathscr{L}_{x,\bullet}$ is of type $P_x$ (resp. $[P_x]$) for each $x \in \boldsymbol{D}$. Moreover, denote by $\lambda_x = (\lambda_{x,1},\dots, \lambda_{x,n_x})$ the associated partition, and $\boldsymbol\lambda=\{\lambda_x, x \in \boldsymbol{D}\}$ the collection of partitions.

Two quasi-parabolic local systems $\mathscr{L}_\bullet$ and $\mathscr{L}'_\bullet$ are \emph{isomorphic} if there exists an isomorphism $\phi: (\mathscr{L},\rho) \rightarrow (\mathscr{L}',\rho')$ of local systems such that it is compatible with representations and preserves the quasi-parabolic structures. More precisely, we have
\begin{align*}
    \phi \circ \rho = \rho' \circ \phi, \quad \phi(\mathscr{L}_{x,i}) = \mathscr{L}'_{x,i}.
\end{align*}
\end{defn}

In this section, we are mostly interested in quasi-parabolic local systems of type $[\boldsymbol{P}]$ and thus we suppose that the maximal torus $T$ of diagonal matrices is included in every $P_x$ for $x \in \boldsymbol{D}$.

Since we always consider isomorphism classes of quasi-parabolic local systems, we can fix a vector space $\mathscr{L}$ together with a fixed basis. Therefore, a local system $(\mathscr{L},\rho)$ can be considered as a vector space $\mathscr{L}=k^n$ together with a representation $\rho : \pi_1(X_{\boldsymbol{D}}) \rightarrow {\rm GL}_n(k)$.

\begin{lem}\label{lem_quasi_para_and_orbit}
Given a collection $\boldsymbol{P}$ of parabolic subgroups in ${\rm GL}(\mathscr{L})$, there is a one-to-one correspondence between isomorphism classes of quasi-parabolic local systems of type $[\boldsymbol{P}]$ on $X_{\boldsymbol{D}}$ and ${\rm GL}(\mathscr{L})$-orbits of fundamental group representations in ${\rm Hom}(\pi_1(X_{\boldsymbol{D}}), {\rm GL}(\mathscr{L}))$ such that each orbit includes a point $[\rho]$ satisfying the condition that $\rho(c_x)$ is conjugate to an element in $P_{x}$ for $x \in \boldsymbol{D}$.
\end{lem}

\begin{proof}
By Definition \ref{defn_quasi_parab_lc}, it is easy to find that isomorphism classes of quasi-parabolic local systems (with underlying vector space $\mathscr{L}$) correspond to orbits of representations in ${\rm Hom}(\pi_1(X_{\boldsymbol{D}}), {\rm GL}(\mathscr{L}))$. The condition $\rho(c_x)(\mathscr{L}_{x,i}) = \mathscr{L}_{x,i}$ implies that the element $\rho(c_x)$ is conjugate to an element in $P_{x}$. On the other hand, given a representation $\rho \in {\rm Hom}(\pi_1(X_{\boldsymbol{D}}), {\rm GL}(\mathscr{L}))$ such that $\rho(c_x) \in g_x P_{x} g_x^{-1}$ for some $g_x \in {\rm GL}(\mathscr{L})$. Note that the parabolic subgroup $g_x P_{x} g_x^{-1}$ determines a unique quasi-parabolic structure of $\mathscr{L}$ on the puncture $x$, and $g_x P_x g^{-1}_x$ is regarded as the automorphism group of the filtration. Clearly, $P_x$ is conjugate to $ g_x P_x g^{-1}_x$. Thus, we obtain a quasi-parabolic local system of type $[\boldsymbol{P}]$.
\end{proof}

With the same idea as in \S\ref{subsect_rep_fund_grp}, we can construct a quiver such that representations of the fundamental group $\pi_1(X_{\boldsymbol{D}})$ are equivalent to representations of such a quiver, and then we obtain a correspondence between quasi-parabolic local systems on $X_{\boldsymbol{D}}$ and quiver representations. This quiver will be constructed after we introduce stability conditions because it is important for our purposes that this one-to-one correspondence also preserves the stability conditions.
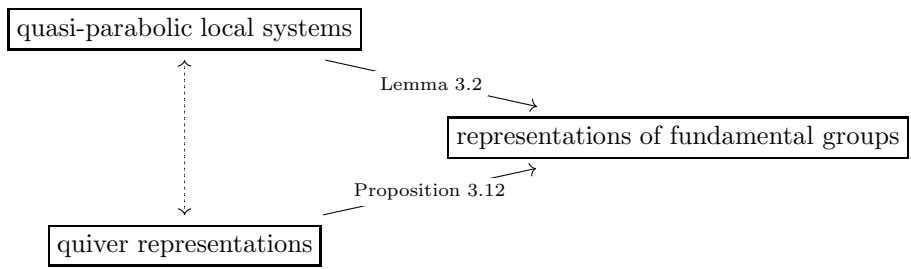
\begin{center}
\begin{tikzcd}
\fbox{quasi-parabolic local systems} \arrow[dd, dotted] \arrow[rd, "\text{Lemma }\ref{lem_quasi_para_and_orbit}" description] & \\
& \fbox{representations of fundamental groups} \\
\fbox{quiver representations} \arrow[uu, dotted] \arrow[ru,"\text{Proposition }\ref{prop_para_and_orbit}" description] &
\end{tikzcd}
\end{center}

\subsection{Filtered (Parabolic) Local Systems}
Now we move to parabolic structures and stability conditions. A \emph{parabolic structure} on $V$ is a quasi-parabolic structure
\begin{align*}
V = V_1 \supsetneq V_2 \supsetneq \dots \supsetneq V_{n'+1}  = \{0\},
\end{align*}
and each $V_i$ is associated to a rational number $\theta_i$ for $1 \leq i \leq n'$, where $\theta_i$ are distinct numbers. The collection $\{\theta_i, 1 \leq i \leq n' \}$ is called the \emph{weights}. In sum, a parabolic structure on $V$ can be regarded as a weighted filtration.

\begin{defn}\label{defn_fls}
A \emph{filtered (parabolic) local system $\mathscr{L}_{\bullet}$} is a local system $(\mathscr{L},\rho)$ equipped with a parabolic structure for each puncture $x \in \boldsymbol{D}$. More precisely, for each puncture $x \in \boldsymbol{D}$, denote by
\begin{align*}
\mathscr{L}_{x,\bullet}: \mathscr{L} = \mathscr{L}_{x,1} \supsetneq \dots \supsetneq \mathscr{L}_{x,n_x+1} = \{0\}
\end{align*}
the quasi-parabolic structure on the puncture $x$, where $n_x$ is a positive integer, and denote by $\theta_{x,i}$ the rational number associated to $\mathscr{L}_{x,i}$. Let $\boldsymbol{P}=\{P_x, x \in \boldsymbol{D}\}$ be a collection of parabolic subgroups in ${\rm GL}(\mathscr{L})$. If $P_x = {\rm Aut}(\mathscr{L}_{x,\bullet})$ (resp. $P_x$ is conjugate to ${\rm Aut}(\mathscr{L}_{x,\bullet})$), then it is called \emph{of type $\boldsymbol{P}$} (resp. \emph{$[\boldsymbol{P}]$}) \emph{with weights $\boldsymbol\theta$}, where $\boldsymbol\theta=\{\theta_{x,i}, x \in \boldsymbol{D}, 1 \leq i \leq n_x\}$. For each $x \in \boldsymbol{D}$, we have $\theta_{x,i} \neq \theta_{x,j}$ whenever $i \neq j$.

Two filtered local systems $\mathscr{L}_\bullet$ and $\mathscr{L}'_\bullet$ are \emph{isomorphic} if there exists an isomorphism $\phi: \mathscr{L}_\bullet \rightarrow \mathscr{L}'_\bullet$ of quasi-parabolic local systems such that the corresponding weights are the same $\theta_{x,i} = \theta'_{x,i}$.
\end{defn}

\begin{rem}
The definition of filtered (parabolic) local systems is similar to that of parabolic bundles introduced in \cite{MS80}. However, in this paper, we use the terminology \emph{filtered local system} rather than \emph{parabolic local system} because Simpson first defined and studied this type of local systems in \cite{Sim90} and used the notion ``filtered local system". Moreover, in \cite{MS80}, the weights for parabolic bundles are assumed to be non-negative rational numbers in $[0,1)$ in increasing order, i.e.
\begin{align*}
0 \leq \theta_{x,1} < \dots < \theta_{x,n_x} < 1.
\end{align*}
However, when we consider filtered local systems, it is not necessary to make these assumptions. The reason for this setup comes from the result of tame nonabelian Hodge correspondence on noncompact curves, where the weights of a filtered local system are determined by the real part of the eigenvalues of the residue of the corresponding Higgs field (see \cite[\S 5]{Sim90}).
\end{rem}

\begin{defn}
Given a filtered local system $\mathscr{L}_\bullet$ of type $[\boldsymbol{P}]$ with weights $\boldsymbol\theta$, denote by $\boldsymbol\lambda$ the associated collection of partitions. The \emph{degree} of $\mathscr{L}_\bullet$ is defined as
\begin{align*}
    \deg^{\rm loc} (\mathscr{L}_\bullet):= \sum_{x \in \boldsymbol{D}} \sum_{i=1}^{n_x} \theta_{x,i} \cdot \dim(\mathscr{L}_{x,i}/\mathscr{L}_{x,i+1}) = \sum_{x \in \boldsymbol{D}} \sum_{i=1}^{n_x} \theta_{x,i} \lambda_{x,i}.
\end{align*}
\end{defn}

Similar to the definition of parabolic sub-bundles \cite[Definition 1.7]{MS80}, we give the definition of filtered sub-local systems.
\begin{defn}
A filtered local system $\mathscr{L}'_\bullet$ is a \emph{filtered sub-local system} of a filtered local system $\mathscr{L}_\bullet$ (of type $[\boldsymbol{P}]$ with weights $\boldsymbol\theta$), if
\begin{itemize}
\item the underlying local system $(\mathscr{L}',\rho')$ is a sub-local system of $(\mathscr{L},\rho)$. More precisely, $\mathscr{L}'$ is a subspace of $\mathscr{L}$ which is preserved by $\rho$ and $\rho'$ is the restriction $\rho|_{\mathscr{L}'}$;
\item for each puncture $x \in \boldsymbol{D}$, given $1 \leq j \leq n'(x)$, let $i$ be the largest integer such that $\mathscr{L}'_{x,j} \subseteq \mathscr{L}_{x,i}$, and then $\theta'_{x,j} = \theta_{x,i}$.
\end{itemize}
\end{defn}

\begin{defn}\label{defn_fls_stab}
A filtered local system $\mathscr{L}_\bullet$ is \emph{semistable} (resp. \emph{stable}), if for any nontrivial proper filtered sub-local system $\mathscr{L}'_\bullet \subseteq \mathscr{L}_\bullet$, we have
\begin{align*}
\frac{\deg^{\rm loc} (\mathscr{L}'_\bullet)}{{\rm rk}(\mathscr{L}'_\bullet)} \leq \frac{\deg^{\rm loc} (\mathscr{L}_\bullet)}{{\rm rk}(\mathscr{L}_\bullet)} \quad (\text{resp. } <).
\end{align*}
\end{defn}

\begin{exmp}\label{exmp_sub_local_system}
In this example, we calculate the degree of nontrivial filtered sub-local systems. Let $\mathscr{L}'_\bullet \subseteq \mathscr{L}_\bullet$ be a filtered sub-local system. Note that $0 \subsetneq \mathscr{L}' \subsetneq \mathscr{L}$ is a filtration of vector spaces. Then, the degree of $\mathscr{L}'_\bullet$ is
\begin{align*}
\deg^{\rm loc}(\mathscr{L}'_\bullet) = \sum_{x \in \boldsymbol{D}} \sum_{i=1}^{n_x} \theta_{x,i} \cdot \dim( \mathscr{L}' \cap \mathscr{L}_{x,i} / \mathscr{L}_{x,i+1}  ).
\end{align*}
Moreover, we denote by $P_{\mathscr{L}'} \subseteq {\rm GL}(\mathscr{L})$ the automorphism group of the filtration $\mathscr{L} \supsetneq \mathscr{L}' \supsetneq 0$, which is clearly a parabolic subgroup.
\end{exmp}

\begin{lem}\label{lem_par_deg_zero}
Let $\mathscr{L}_\bullet$ be a filtered local system of type $[\boldsymbol{P}]$ with weights $\boldsymbol\theta$. Suppose that $\deg^{\rm loc}(\mathscr{L}_\bullet) = 0$. Then, $\mathscr{L}_\bullet$ is semistable (stable) if and only if for any nontrivial proper filtered sub-local system $\mathscr{L}'_\bullet \subseteq \mathscr{L}_\bullet$, we have
\begin{align*}
    \deg^{\rm loc}(\mathscr{L}'_\bullet) \leq 0 \quad (\text{resp. } <).
\end{align*}
\end{lem}

\begin{proof}
If $\deg^{\rm loc}(\mathscr{L}_\bullet) = 0$, then $\mathscr{L}_\bullet$ is semistable if and only if we have
\begin{align*}
\frac{\deg^{\rm loc} (\mathscr{L}'_\bullet)}{{\rm rk}(\mathscr{L}'_\bullet)} \leq \frac{\deg^{\rm loc} (\mathscr{L}_\bullet)}{{\rm rk}(\mathscr{L}_\bullet)} =0
\end{align*}
for any nontrivial filtered sub-local system of $\mathscr{L}'_\bullet$. Therefore,
\begin{align*}
\deg^{\rm loc}(\mathscr{L}'_\bullet) \leq 0.
\end{align*}
The proof for the stable case is exactly the same.
\end{proof}

\subsection{The Quiver $\mathscr{Q}^{\boldsymbol{D}}$}

In this subsection, we construct a new quiver $\mathscr{Q}^{\boldsymbol{D}}$ and prove that there is a one-to-one correspondence between filtered local systems and certain representations of this quiver. Furthermore, this correspondence also preserves the stability condition, which will be proved in \S\ref{subsect_equiv_stab_local}. The underlying vector space of a quasi-parabolic local system in this subsection is always considered as $k^n$ under an isomorphism.

\begin{con}\label{cons_new_quiv}
We first define a quiver $\mathscr{Q}^{\boldsymbol{D}} = (\mathscr{Q}^{\boldsymbol{D}}_0,\mathscr{Q}^{\boldsymbol{D}}_1)$ as follows:
\begin{align*}
\mathscr{Q}_0^{\boldsymbol{D}} & = \{v_0, v_{x}, x \in \boldsymbol{D}\}, \\
\mathscr{Q}_1^{\boldsymbol{D}} & = \{ a_i,b_j: v_0 \rightarrow v_0, 1 \leq i,j \leq g, \\
& c_{x,1}: v_{0} \rightarrow v_{x}, \text{ } c_{x,2}: v_x \rightarrow v_0, \text{ } x \in \boldsymbol{D}\}.
\end{align*}
More precisely, for each point $x \in \boldsymbol{D}$, we have
\begin{center}
\begin{tikzcd}
v_0 \arrow[rr, bend right, "c_{x,1}" description] & & v_x \arrow[ll, bend right, "c_{x,2}" description]
\end{tikzcd}
\end{center}
and denote by $c_x:=c_{x,2}c_{x,1}$ the loop with source and target $v_0$\footnote{The quiver we are considering has some similarities to the ``comet-shaped quivers'' considered in \cite{Bal23,CBS06,HLRV11}. However, in these papers, the authors only fix the monodromies for each puncture, without any parabolic structure, hence their moduli spaces exist as affine GIT quotients. We thank Markus Reineke for pointing out these wonderful references.}. Define a relation
\begin{align*}
    \mathscr{I}^{\boldsymbol{D}}:=(\prod_{i=1}^g [a_i,b_i]) \cdot (\prod_{x\in \boldsymbol{D}} c_x) - v_0 .
\end{align*}
Then, the $(\prod_{v \in \mathscr{Q}^{\boldsymbol{D}}_0} {\rm GL}_n(k))$-orbits in $\mathscr{R}(\mathscr{Q}^{\boldsymbol{D}},\mathscr{I}^{\boldsymbol{D}},n)$ correspond to rank $n$ quiver representations of $\mathscr{Q}^{\boldsymbol{D}}$ satisfying the relation $\mathscr{I}^{\boldsymbol{D}}$. 

We fix a collection of parabolic subgroups $\boldsymbol{P}=\{P_x, x \in \boldsymbol{D}\}$ in ${\rm GL}_n(k)$. For each $x\in\boldsymbol{D}$, denote by $L_{x}$ the Levi subgroup of $P_x$. We construct a quasi-projective variety $\mathscr{R}(\mathscr{Q}^{\boldsymbol{D}},\mathscr{I}^{\boldsymbol{D}},[\boldsymbol{P}])$ as follows. We first define a locally closed subset
\begin{align*}
    \mathscr{R}(\mathscr{Q}^{\boldsymbol{D}},\boldsymbol{P}) \subseteq \mathscr{R}(\mathscr{Q}^{\boldsymbol{D}},n),
\end{align*}
whose elements $\phi= (\phi_a)_{a \in \mathscr{Q}^{\boldsymbol{D}}_1}$ satisfy the conditions that $\phi_a \in {\rm GL}_n(k)$ for every $a \in \mathscr{Q}_1^{\boldsymbol{D}}$ and 
\begin{align*}
    \phi_{c_{x,1}} \in L_{x}, \quad \phi_{c_{x,2}} \in P_{x}
\end{align*}
for $x \in \boldsymbol{D}$. Consider the group $\prod_{x \in \boldsymbol{D}} G_x$, where $G_x = {\rm GL}_n(k)$, and we define a $(\prod_{x \in \boldsymbol{D}} G_x)$-action on $\mathscr{R}(\mathscr{Q}^{\boldsymbol{D}},n)$ as follows
\begin{align*}
    (g_x) \cdot (\phi_{a_i},\phi_{b_j},\phi_{c_{x,1}},\phi_{c_{x,2}}) := (\phi_{a_i},\phi_{b_j},  \phi_{c_{x,1}}g_x^{-1}, g_x \phi_{c_{x,2}}).
\end{align*}
Denote by $\mathscr{R}'(\mathscr{Q}^{\boldsymbol{D}},\boldsymbol{P})$ the fiber product
\begin{center}
\begin{tikzcd}
\mathscr{R}'(\mathscr{Q}^{\boldsymbol{D}},\boldsymbol{P}) \arrow[r, dotted] \arrow[d,dotted] & \mathscr{R}(\mathscr{Q}^{\boldsymbol{D}},\boldsymbol{P}) \arrow[d] \\
(\prod_{x \in \boldsymbol{D}} G_x) \times \mathscr{R}(\mathscr{Q}^{\boldsymbol{D}},n) \arrow[r] & \mathscr{R}(\mathscr{Q}^{\boldsymbol{D}},n) \ .
\end{tikzcd}
\end{center}
Clearly, $\mathscr{R}'(\mathscr{Q}^{\boldsymbol{D}},\boldsymbol{P})$ includes all tuples
\begin{align*}
    ( (g_x) , (\phi_{a_i},\phi_{b_j},\phi_{c_{x,1}},\phi_{c_{x,2}}) ) \in (\prod_{x \in \boldsymbol{D}} G_x) \times \mathscr{R}(\mathscr{Q}^{\boldsymbol{D}},n)
\end{align*}
such that
\begin{align*}
    \phi_{c_{x,1}} g^{-1}_x \in L_x, \ \ g_x \phi_{c_{x,2}} \in P_x.
\end{align*}
Therefore, the projection
\begin{align*}
    \mathscr{R}(\mathscr{Q}^{\boldsymbol{D}},[\boldsymbol{P}]) := \mathscr{R}'(\mathscr{Q}^{\boldsymbol{D}},\boldsymbol{P})|_{\mathscr{R}(\mathscr{Q}^{\boldsymbol{D}},n)}
\end{align*}
parametrizes tuples $(\phi_{a_i},\phi_{b_j},\phi_{c_{x,1}},\phi_{c_{x,2}})$ such that for each $x \in \boldsymbol{D}$, there exists $g_x \in G_x = {\rm GL}_n(k)$ satisfying 
\begin{align*}
    \phi_{c_{x,1}} g^{-1}_x \in L_x, \ \ g_x \phi_{c_{x,2}} \in P_x.
\end{align*}
Now we add the relation $\mathscr{I}^{\boldsymbol{D}}$ and obtain a closed subvariety $\mathscr{R}(\mathscr{Q}^{\boldsymbol{D}},\mathscr{I}^{\boldsymbol{D}},[\boldsymbol{P}]) \subseteq \mathscr{R}(\mathscr{Q}^{\boldsymbol{D}},[\boldsymbol{P}])$. Clearly, $\mathscr{R}(\mathscr{Q}^{\boldsymbol{D}},\mathscr{I}^{\boldsymbol{D}},[\boldsymbol{P}])$ is quasi-projective.

Now we introduce a reductive group $G_{\boldsymbol{P}}$ together with an action on $\mathscr{R}(\mathscr{Q}^{\boldsymbol{D}},\mathscr{I}^{\boldsymbol{D}},[\boldsymbol{P}])$. Define
\begin{align*}
    G_{\boldsymbol{P}} = \prod_{v \in \mathscr{Q}^{\boldsymbol{D}}_0} G_v := {\rm GL}_n(k) \times \prod_{x \in \boldsymbol{D}} L_{x},
\end{align*}
where $G_{v_0}= {\rm GL}_n(k)$ and $G_{v_x} = L_{x}$. As a subgroup of $\prod_{v \in \mathscr{Q}^{\boldsymbol{D}}_0} {\rm GL}_n(k)$, we have a natural action of $G_{\boldsymbol{P}}$ on $\mathscr{R}(\mathscr{Q}^{\boldsymbol{D}},n)$. Note that the $(\prod_{x \in \boldsymbol{D}} L_x)$-action on $\mathscr{R}(\mathscr{Q}^{\boldsymbol{D}},\boldsymbol{P})$ is given by
\begin{align*}
    (g_{v_x}) \cdot (\phi_{a_i}, \phi_{b_j}, \phi_{c_{x,1}}, \phi_{c_{x,2}}) = (\phi_{a_i}, \phi_{b_j}, g_{v_x} \phi_{c_{x,1}} , \phi_{c_{x,2}} g_{v_x}^{-1}),
\end{align*}
where $\phi_{c_{x,1}} \in L_x$, $\phi_{c_{x,2}} \in P_x$ and $g_{v_x} \in L_x$, which is well-defined. Thus, we obtain a well-defined $G_{\boldsymbol{P}}$-action on $\mathscr{R}(\mathscr{Q}^{\boldsymbol{D}},\mathscr{I}^{\boldsymbol{D}},[\boldsymbol{P}])$.

At the end of the construction, we would like to emphasize that the group $\prod_{x \in \boldsymbol{D}} G_x$ is used to construct $\mathscr{R}(\mathscr{Q}^{\boldsymbol{D}},\mathscr{I}^{\boldsymbol{D}},[\boldsymbol{P}])$, while the group $G_{\boldsymbol{P}}$ will be used later to introduce stability conditions.
\end{con}

\begin{rem}
In \cite{AD20}, the authors studied quiver representations with parabolic structures by constructing a new quiver, called the \emph{ladder quiver}. Their approach is interesting, but it is hard to apply it directly to our case. The main problem is that the quiver we consider includes loops, which comes from fundamental group representations.
\end{rem}

\begin{prop}\label{prop_para_and_orbit}
There is a one-to-one correspondence between $G_{\boldsymbol{P}}$-orbits in $\mathscr{R}(\mathscr{Q}^{\boldsymbol{D}},\mathscr{I}^{\boldsymbol{D}},[\boldsymbol{P}])$ and isomorphism classes of quasi-parabolic local systems of type $[\boldsymbol{P}]$ on $X_{\boldsymbol{D}}$.
\end{prop}

\begin{proof}
By Borel covering theorem, given a parabolic subgroup $P$, any element in $G$ is conjugate to an element in $P$. Then, Lemma \ref{lem_quasi_para_and_orbit} shows that there is a one-to-one correspondence between isomorphism classes of quasi-parabolic local systems of type $[\boldsymbol{P}]$ and ${\rm GL}_n(k)$-orbits in ${\rm Hom}(\pi_1(X_{\boldsymbol{D}}) , {\rm GL}_n(k))$. Thus, it is equivalent to prove that ${\rm GL}_n(k)$-orbits in ${\rm Hom}(\pi_1(X_{\boldsymbol{D}}) , {\rm GL}_n(k))$ are in one-to-one correspondence with $G_{\boldsymbol{P}}$-orbits in $\mathscr{R}(\mathscr{Q}^{\boldsymbol{D}},\mathscr{I}^{\boldsymbol{D}},[\boldsymbol{P}])$.

Given a point $(\phi_a)_{a \in \mathscr{Q}^{\boldsymbol{D}}_1} \in \mathscr{R}(\mathscr{Q}^{\boldsymbol{D}},\mathscr{I}^{\boldsymbol{D}},[\boldsymbol{P}])$, we suppose that 
\begin{align*}
    \phi_{c_{x,1}} g_x^{-1} \in  L_{x}, \quad g_x \phi_{c_{x,2}} \in P_{x}
\end{align*}
with some $g_x \in {\rm GL}_n(k)$ for each $x \in \boldsymbol{D}$. We associate this element with a representation $\rho: \pi_1(X_{\boldsymbol{D}}) \rightarrow {\rm GL}_n(k)$ as follows
\begin{align*}
    \rho(a_i) = \phi_{a_i}, \quad \rho(b_j) = \phi_{b_j}, \quad \rho(c_x) = \phi_{c_x} = \phi_{c_{x,2}} \phi_{c_{x,1}}.
\end{align*}
Clearly,  $\rho(c_x) \in (g^{-1}_x) P_{x} (g^{-1}_x)^{-1}$ for each $x \in \boldsymbol{D}$, and then the representation $\rho$ corresponds to a quasi-parabolic local system of type $[\boldsymbol{P}]$. Moreover, this induces a natural surjection
\begin{align*}
    \mathscr{R}(\mathscr{L}^{\boldsymbol{D}},\mathscr{I}^{\boldsymbol{D}} , [\boldsymbol{P}]) \rightarrow {\rm Hom}(\pi_1(X_{\boldsymbol{D}}), {\rm GL}_n(k)).
\end{align*}

Now let $(g_v)_{v \in \mathscr{Q}^{\boldsymbol{D}}_0} \in G_{\boldsymbol{P}}$. We consider the point
\begin{align*}
    (\phi'_a)_{a \in \mathscr{Q}^{\boldsymbol{D}}_1} = (g_v)_{v \in \mathscr{Q}^{\boldsymbol{D}}_0} \cdot (\phi_a)_{a \in \mathscr{Q}^{\boldsymbol{D}}_1},
\end{align*}
and we have
\begin{align*}
    \phi'_{a_i} = g_{v_0} \phi_{a_i} g^{-1}_{v_0}, \ \phi'_{b_j} = g_{v_0} \phi_{b_j} g^{-1}_{v_0}, \ \phi'_{c_{x,1}} = g_{v_x} \phi_{c_{x,1}} g^{-1}_{v_0}, \ \phi'_{c_{x,2}} = g_{v_0} \phi_{c_{x,2}} g^{-1}_{v_x}.
\end{align*}
Denote by $\rho': \pi_1(X_{\boldsymbol{D}}) \rightarrow {\rm GL}_n(k)$ the corresponding representation of $(\phi'_a)_{a \in \mathscr{Q}^{\boldsymbol{D}}_1}$. Clearly,
\begin{align*}
    & \rho'(a_i) = \phi'_{a_i} = g_{v_0} \phi_{a_i} g^{-1}_{v_0} = g_{v_0} \rho(a_i) g^{-1}_{v_0},\\
    & \rho'(b_j) = \phi'_{b_j} = g_{v_0} \phi_{b_j} g^{-1}_{v_0} = g_{v_0} \rho(b_j) g^{-1}_{v_0}, \\
    & \rho'(c_x) = \phi'_{c_{x,2}} \phi'_{c_{x,1}} = (g_{v_0} \phi_{c_{x,2}} g^{-1}_{v_x}) (g_{v_x} \phi_{c_{x,1}} g^{-1}_{v_0}) = g_{v_0} \rho(c_x) g^{-1}_{v_0}.
\end{align*}
Therefore,
\begin{align*}
    g_{v_0} \rho g^{-1}_{v_0} = \rho'.
\end{align*}
In conclusion, we have a one-to-one correspondence between $G_{\boldsymbol{P}}$-orbits in $\mathscr{R}(\mathscr{Q}^{\boldsymbol{D}},\mathscr{I}^{\boldsymbol{D}},[\boldsymbol{P}])$ and ${\rm GL}_n(k)$-orbits in ${\rm Hom}(\pi_1(X_{\boldsymbol{D}}),{\rm GL}_n(k))$.
\end{proof}

\begin{exmp}\label{exmp_sub_lc_qrep}
Let $\mathscr{L}_\bullet$ be a quasi-parabolic local system of type $[\boldsymbol{P}]$, and denote by $\phi = (\phi_a)_{a \in \mathscr{Q}^{\boldsymbol{D}}_1} \in \mathscr{R}(\mathscr{Q}^{\boldsymbol{D}},\mathscr{I}^{\boldsymbol{D}},[\boldsymbol{P}])$ a point corresponding to $\mathscr{L}_\bullet$ by Proposition \ref{prop_para_and_orbit}.

Let $\mathscr{L}'_\bullet$ be a quasi-parabolic sub-local system of $\mathscr{L}_\bullet$. As a quasi-parabolic local system, $\mathscr{L}'_\bullet$ corresponds to a point in the representation space of $\mathscr{Q}^{\boldsymbol{D}}$ and denote it by $\phi'$. By the proof of Proposition \ref{prop_para_and_orbit}, it is clear that $\phi' = \phi|_{\mathscr{L}'}$. This property implies that $\phi_a \in P_{\mathscr{L}'}$ for each arrow $a$, where $P_{\mathscr{L}'}$ is defined by the subspace $\mathscr{L}' \subseteq \mathscr{L}$ as in Example \ref{exmp_sub_local_system}.

Now let $P \in {\rm GL}_n(k)$ be a parabolic subgroup and denote by $\lambda = (\lambda_1,\dots,\lambda_{n'})$ the associated partition of $P$. We suppose that $\phi_a \in P$ for all arrows $a$. As studied in \S\ref{subsect_qplc}, let
\begin{align*}
\mathscr{L} = \mathscr{L}_{1} \supsetneq \dots \supsetneq \mathscr{L}_{n'+1} = \{0\}
\end{align*}
be the filtration of $\mathscr{L}$ given by $P$. Since $\phi_a \in P$ for all arrows $a$, the restriction of the representation $\rho|_{\mathscr{L}_j}$ is well-defined. Therefore, the filtration of $\mathscr{L}$ induces a filtration of quasi-parabolic sub-local systems
\begin{align*}
\mathscr{L}_\bullet = \mathscr{L}_{1,\bullet} \supsetneq \dots \supsetneq \mathscr{L}_{n'+1,\bullet} = \{0\},
\end{align*}
and $(\mathscr{L}_j,\rho|_{\mathscr{L}_j})$ is the underlying local system of $\mathscr{L}_{j,\bullet}$.
\end{exmp}

\subsection{Equivalence of Stability Conditions}\label{subsect_equiv_stab_local}
We follow the same notations as in previous sections. Let $\boldsymbol{P}=\{P_x, x \in \boldsymbol{D}\}$ be a collection of parabolic subgroups. Let $P_{x} = L_{x} U_{x}$ be the Levi decomposition of the parabolic subgroup $P_{x}$, where $L_{x}$ is the Levi subgroup and $U_{x}$ is the unipotent group, and we write $L_{x} = \prod_{i=1}^{n_x} L_{x,i}$ as a product of general linear groups. Recall that
\begin{align*}
    G_{\boldsymbol{P}}={\rm GL}_{n}(k) \times \prod_{x \in \boldsymbol{D}} L_{x}.
\end{align*}
Let $\mu$ be a cocharacter of $G_{\boldsymbol{P}}$, i.e. $\mu=(\mu_v)_{v \in \mathscr{Q}^{\boldsymbol{D}}_0}$, where $\mu_{v_0}$ is a cocharacter of ${\rm GL}_n(k)$ and $\mu_{v_x}$ is a cocharacter of $L_{x}$. Since $T$ is included in $P_{x}$, all cocharacters $\mu_v$ can be regarded as cocharacters of the maximal torus $T$.

\begin{lem}\label{lem_pre_mu}
Let $\mu_1,\mu_2$ be two cocharacters of a maximal torus $T \subseteq {\rm GL}_n(k)$. Define
\begin{align*}
    & S_{12}:=\{g \in {\rm End}(k^n) \text{ } | \text{ } \lim_{t\rightarrow 0} \mu_{1}(t) \cdot g \cdot \mu_{2}^{-1}(t) \text{ exists.}\} \\
    & S_{21}:=\{g \in {\rm End}(k^n) \text{ } | \text{ } \lim_{t\rightarrow 0} \mu_{2}(t) \cdot g \cdot \mu_{1}^{-1}(t) \text{ exists.}\}
\end{align*}
Then, $S_{12} \cap {\rm GL}_n(k) \neq \emptyset$ and $S_{21} \cap {\rm GL}_n(k) \neq \emptyset$ if and only if $\mu_1 = \mu_2$.
\end{lem}

\begin{proof}
Applying the argument in \cite[Chapter 2, \S 2]{MFK94}, this lemma follows directly.
\end{proof}

\begin{lem}\label{lem_cochar_equal}
Let $\phi = (\phi_a)_{a \in \mathscr{Q}^{\boldsymbol{D}}_1} \in \mathscr{R}(\mathscr{Q}^{\boldsymbol{D}},\mathscr{I}^{\boldsymbol{D}}, [\boldsymbol{P}])$ be a point, and for each $x \in \boldsymbol{D}$, and let $\mu=(\mu_v)_{v \in \mathscr{Q}^{\boldsymbol{D}}_0}$ be a cocharacter of $G_{\boldsymbol{P}}$ such that the limit $\lim_{t \rightarrow 0} \mu(t) \cdot \phi$ exists. Then for each $x \in \boldsymbol{D}$, there exists $g_x \in {\rm GL}_n(k)$ such that
\begin{itemize}
    \item $g_x \phi_{c_{x,2}} \in P_{x}$ and $\phi_{c_{x,1}} g^{-1}_x \in L_x$,
    \item $\mu_{v_x}(t) = g_x \mu_{v_0}(t) g^{-1}_x$.
\end{itemize}
\end{lem}

\begin{proof}
The limit $\lim_{t \rightarrow 0} \mu(t) \cdot \phi$ exists if and only if the limit
\begin{align*}
    \lim_{t\rightarrow 0} \mu_{t(a)}(t) \phi_a \mu_{s(a)}^{-1}(t)
\end{align*}
exists in ${\rm GL}_n(k)$ for each arrow $a \in \mathscr{Q}^{\boldsymbol{D}}_1$. Then, for each $x \in \boldsymbol{D}$, the limits
\begin{align*}
\lim_{t\rightarrow 0} \mu_{v_x}(t) \phi_{c_{x,1}} \mu_{v_0}^{-1}(t), \ \ \lim_{t\rightarrow 0} \mu_{v_0}(t) \phi_{c_{x,2}} \mu_{v_x}^{-1}(t)
\end{align*}
exist, which implies the existence of the limit
\begin{align*}
    \lim_{t\rightarrow 0} \mu_{v_0}(t) (\phi_{c_{x,2}} \phi_{c_{x,1}}) \mu_{v_0}^{-1}(t).
\end{align*}
By definition, there exists $g'_x \in {\rm GL}_n(k)$ such that $g'_x \phi_{c_{x,2}} \in P_x$ and $\phi_{c_{x,1}} g'^{-1}_x \in L_x$. Moreover, $g'_x \phi_{c_{x,2}}\phi_{c_{x,1}} g'^{-1}_x \in P_x$. The existence of the limit
\begin{align*}
    \lim_{t\rightarrow 0} (g'_x \mu_{v_0}(t) g'^{-1}_x) (g'_x \phi_{c_{x,2}} \phi_{c_{x,1}} g'^{-1}_x) (g'_x \mu_{v_0}^{-1}(t) g'^{-1}_x)
\end{align*}
implies that $g'_x \mu_{v_0}(t) g'^{-1}_x$ is a cocharacter of some maximal torus in $P_x$ (and thus $L_x$). Note that $\mu_{v_x}$ is a cocharacter of $L_x$ because the normalizer $P_x$ is itself. Therefore, we can suppose that there exists $l_x \in L_x$ such that $(l_x g'_x) \mu_{v_0}(t) (l_x g'_x)^{-1}$ and $\mu_{v_x}(t)$ are cocharacters of the same maximal torus in $L_x$. Then we define $g_x: = l_x g'_x$, and clearly, $g_x \phi_{c_{x,2}} \in P_x$ and $\phi_{c_{x,1}} g^{-1}_x \in L_x$. Now we consider the limits
\begin{align*}
    & g_x (\lim_{t\rightarrow 0} \mu_{v_0}(t) \phi_{c_{x,2}} \mu_{v_x}^{-1}(t) )= \lim_{t\rightarrow 0} (g_x \mu_{v_0}(t) g^{-1}_x) (g_x \phi_{c_{x,2}}) \mu_{v_x}^{-1}(t)\\
    & (\lim_{t\rightarrow 0} \mu_{v_x}(t) \phi_{c_{x,1}} \mu_{v_0}^{-1}(t)) g^{-1}_x =  (\lim_{t\rightarrow 0} \mu_{v_x}(t) (\phi_{c_{x,1}}g_x^{-1})  (g_x  \mu_{v_0}^{-1}(t) g^{-1}_x) 
\end{align*}
Since $g_x \phi_{c_{x,2}} \in P_x$, $\phi_{c_{x,1}} g^{-1}_x \in L_x$ and both $g_x \mu_{v_0}(t) g^{-1}_x$ and $\mu_{v_x}(t)$ are cocharacters of the same maximal torus in $L_x$, the existence of the above limits gives $g_x \mu_{v_0} g^{-1}_x= \mu_{v_x}$ by Lemma \ref{lem_pre_mu} for each $x \in \boldsymbol{D}$.
\end{proof}

The above lemma actually implies that given a point $\phi$, if the limit $\lim_{t \rightarrow 0} \mu(t) \cdot \phi$ exists, then the cocharacter $\mu$ is determined by $\mu_{v_0}: \mathbb{G}_m \rightarrow {\rm GL}_n(k)$ and $\mu_{v_0}$ is conjugate to $\mu_{v_x}$ when regarded as cocharacters of ${\rm GL}_n(k)$. Abusing the notation, we define a parabolic subgroup 
\begin{align*}
    P_\mu := P_{\mu_{v_0}} = \{ g \in {\rm GL}_n(k) \text{ } | \text{ } \lim_{t \rightarrow 0} \mu_{v_0}(t)  g \mu^{-1}_{v_0}(t) \text{ exists}  \}.
\end{align*}

\begin{lem}\label{lem_p_mu}
Suppose that the limit $\lim_{t \rightarrow 0} \mu(t) \cdot \phi$ exists. Then the corresponding representation $\rho:\pi_1(X_{\boldsymbol{D}}) \rightarrow {\rm GL}_n(k)$ is well-defined when restricted to $P_\mu$.
\end{lem}

\begin{proof}
The limit $\lim_{t \rightarrow 0} \mu(t) \cdot \phi$ exists if and only if the following limits exist:
\begin{align*}
    & \lim_{t\rightarrow 0} \mu_{v_0}(t) \phi_{a_i}  \mu_{v_0}^{-1}(t), \ \ \ \lim_{t\rightarrow 0} \mu_{v_0}(t) \phi_{b_j} \mu_{v_0}^{-1}(t), \\
    & \lim_{t\rightarrow 0} \mu_{v_x}(t) \phi_{c_{x,1}} \mu_{v_0}^{-1}(t), \ \lim_{t\rightarrow 0} \mu_{v_0}(t)  \phi_{c_{x,2}} \mu_{v_x}^{-1}(t).
\end{align*}
Since
\begin{align*}
    \rho(a_i) = \phi_{a_i}, \ \rho(b_j) = \phi_{b_j}, \rho(c_x) = \phi_{c_{x,2}} \phi_{c_{x,1}},
\end{align*}
the limits
\begin{align*}
    \lim_{t\rightarrow 0} \mu_{v_0}(t) \rho(a_i)  \mu_{v_0}^{-1}(t), \ \ \lim_{t\rightarrow 0} \mu_{v_0}(t) \rho(b_j)  \mu_{v_0}^{-1}(t), \ \ \lim_{t\rightarrow 0} \mu_{v_0}(t) \rho(c_x)  \mu_{v_0}^{-1}(t)
\end{align*}
exist, which implies that $\rho(a_i)$, $\rho(b_j)$ and $\rho(c_x)$ are in $P_{\mu}$.
\end{proof}

Now we move to characters. A character of $L_{x}$ can be regarded as a character of $P_{x}$, i.e.
\begin{align*}
    {\rm Hom}(P_{x},\mathbb{G}_m) \cong {\rm Hom}(L_{x}, \mathbb{G}_m) \cong \mathbb{Z}^{n_x}.
\end{align*}
In the next construction, we define a character $\chi_{\boldsymbol\theta}$ of $G_{\boldsymbol{P}}$ based on the given data of weights $\boldsymbol\theta$.

\begin{con}\label{cons_char_theta}
A character $\chi: G_{\boldsymbol{P}} \rightarrow \mathbb{G}_m$ is a product of characters of groups associated to vertices. More precisely,
\begin{align*}
    \chi = \chi_{v_0} \cdot \prod_{x \in \boldsymbol{D}} \chi_{v_x}
\end{align*}
such that
\begin{align*}
    \chi( (g_v)_{v \in \mathscr{Q}^{\boldsymbol{D}}_0} ) = \chi_{v_0}(g_{v_0}) \cdot \prod_{x \in \boldsymbol{D}} \chi_{v_x}( g_{v_x}),
\end{align*}
where
\begin{align*}
    \chi_{v_0}: {\rm GL}_n(k) \rightarrow \mathbb{G}_m, \quad \chi_{v_x}: L_{x} \rightarrow \mathbb{G}_m.
\end{align*}
Now let $\boldsymbol\theta=\{\theta_{x,i}, x \in \boldsymbol{D}, 1 \leq i \leq n_x \}$ be a collection of rational numbers, which are regarded as weights. Since there are only finitely many rational numbers, we can write $\theta_{x,i} = \frac{d_{x,i}}{d}$ as fractions with a common denominator $d$. With respect to the above data, we define a character
\begin{align*}
    \chi_{\boldsymbol\theta}  : G_{\boldsymbol{P}} \rightarrow \mathbb{G}_m
\end{align*}
as follows. We write the character as $\chi_{\boldsymbol\theta}= \chi_{v_0} \cdot \prod_{x \in \boldsymbol{D}} \chi_{v_x}$. Since $\chi_{v_x}$ is a character of $L_{x}$, it is a product of characters of $L_{x,i}$. Therefore, we define
\begin{align*}
    \chi_{v_x}(l_x): =  \prod_{i=1}^{n_x} \det(l_{x,i})^{-d_{x,i}},
\end{align*}
where $l_x \in L_{x}$ and $l_{x,i}$ is the component of $l_x$ in $L_{x,i}$. For $\chi_0: {\rm GL}_n(k) \rightarrow \mathbb{G}_m$, we take the trivial character. This finishes the definition of the character $\chi_{\boldsymbol\theta}  : G_{\boldsymbol{P}} \rightarrow \mathbb{G}_m$.
\end{con}

Now we are ready to prove the equivalence of stability conditions.
\begin{prop}\label{prop_stab_local_sys_and_quiv_rep}
Let $\phi = (\phi_a)_{a \in \mathscr{Q}^{\boldsymbol{D}}_1} \in \mathscr{R}(\mathscr{Q}^{\boldsymbol{D}},\mathscr{I}^{\boldsymbol{D}}, [\boldsymbol{P}])$ be a point. Denote by $\mathscr{L}_\bullet$ the corresponding quasi-parabolic local system. Equipped with weights $\boldsymbol\theta$, we obtain a filtered local system $\mathscr{L}_\bullet$. Then, $\mathscr{L}_\bullet$ is semistable (resp. stable) of degree zero if and only if the point $\phi = (\phi_a)_{a \in \mathscr{Q}^{\boldsymbol{D}}_1} \in \mathscr{R}(\mathscr{Q}^{\boldsymbol{D}},\mathscr{I}^{\boldsymbol{D}},[\boldsymbol{P}])$ is $\chi_{\boldsymbol\theta}$-semistable (resp. $\chi_{\boldsymbol\theta}$-stable).
\end{prop}

\begin{proof}
Let $\mu$ be a cocharacter of $G_{\boldsymbol{P}}$ such that the limit $\lim_{t \rightarrow 0} \mu(t) \cdot \phi$ exists. Then the representation $\rho|_{P_\mu}$ is well-defined by Lemma \ref{lem_p_mu}. We define $P:=P_\mu$ to simplify the notation, and denote by $\lambda=(\lambda_1,\dots,\lambda_{n'})$ the corresponding partition of $P$. Let $P=LU$ be the Levi decomposition of $P$ and let $L=\prod_{j=1}^{n'} L_j$ be a product of general linear groups. Note that a cocharacter of $L$ is equivalent to a cocharacter of a maximal torus $T_L$ of $L$. Then we assume that the restriction $\mu_{v_0}|_{L_j}: \mathbb{G}_m \rightarrow T_L \cap L_j$ is given by $t\mapsto t^{d'_j} \cdot {\rm id}$, where $d'_j$ is an integer. Without loss of generality, we suppose that the integers $d'_j$ for $1 \leq j \leq n'$ are in decreasing order. By Example \ref{exmp_sub_lc_qrep}, the parabolic subgroup $P$ determines a filtration of filtered local systems
\begin{align*}
\mathscr{L}_\bullet = \mathscr{L}_{1,\bullet} \supsetneq \dots \supsetneq \mathscr{L}_{n'+1,\bullet} = \{0\},
\end{align*}
and $(\mathscr{L}_j,\rho|_{\mathscr{L}_j})$ is the underlying local system of $\mathscr{L}_{j,\bullet}$. By Example \ref{exmp_sub_local_system}, the degree of $\mathscr{L}_{j,\bullet}$ is
\begin{align*}
\deg^{\rm loc}(\mathscr{L}_{j,\bullet}) = \sum_{x \in \boldsymbol{D}} \sum_{i=1}^{n_x} \theta_{x,i} \cdot \dim( \mathscr{L}_j \cap \mathscr{L}_{x,i} / \mathscr{L}_{x,i+1}  ).
\end{align*}

We calculate the pairing $\langle \mu, \chi_{\boldsymbol\theta} \rangle$ first. Since $\chi_{v_0}$ is trivial by construction, we have
\begin{align*}
    \langle \mu, \chi_{\boldsymbol\theta} \rangle = \sum_{x \in \boldsymbol{D}} \langle \mu_{v_x}, \chi_{v_x} \rangle.
\end{align*}
By Lemma \ref{lem_cochar_equal}, we know $g_x \mu_{v_0} g^{-1}_x = \mu_{v_x}$, and then
\begin{align*}
    \langle \mu_{v_x}, \chi_{v_x} \rangle = \langle \mu_{v_0}, g^{-1}_x \chi_{v_x} g_x \rangle.
\end{align*}
Note that the filtration determined by $\mu_{v_0}$ is 
\begin{align*}
    \mathscr{L} = \mathscr{L}_1 \supsetneq \cdots \supsetneq \mathscr{L}_{n'+1} = \{ 0 \},
\end{align*}
while the filtration given by $g^{-1}_x \chi_{v_x} g_x$ is
\begin{align*}
    \mathscr{L} = \mathscr{L}_{x,1} \supsetneq \cdots \mathscr{L}_{x,n_{x+1}} = \{ 0\}.
\end{align*}
Combining with weights, we have
\begin{align*}
     \langle \mu_{v_x}, \chi_{v_x} \rangle = \sum_{i=1}^{n_x} \sum_{j=1}^{n'}  (- d'_j \cdot d_{x,i}) \cdot \dim(\mathscr{L}_{j} \cap \mathscr{L}_{x,i} / (\mathscr{L}_{j+1} \cap \mathscr{L}_{x,i} + \mathscr{L}_{j} \cap \mathscr{L}_{x,i+1}) ).
\end{align*}
Therefore,
\begin{align*}
    \langle \mu, \chi_{\boldsymbol\theta} \rangle  & = \sum_{x \in \boldsymbol{D}} \langle \mu_{v_x}, \chi_{v_x} \rangle \\
    & = \sum_{x \in \boldsymbol{D}} \sum_{i=1}^{n_x} \sum_{j=1}^{n'}  (- d'_j \cdot d_{x,i}) \cdot \dim(\mathscr{L}_{j} \cap \mathscr{L}_{x,i} / (\mathscr{L}_{j+1} \cap \mathscr{L}_{x,i} + \mathscr{L}_{j} \cap \mathscr{L}_{x,i+1}) ) \\
    & = - \sum_{j=1}^{n'} d'_{j} \left( \sum_{x \in \boldsymbol{D}} \sum_{i=1}^{n_x} d_{x,i} \cdot \dim(\mathscr{L}_{j} \cap \mathscr{L}_{x,i} / (\mathscr{L}_{j+1} \cap \mathscr{L}_{x,i} + \mathscr{L}_{j} \cap \mathscr{L}_{x,i+1}) ) \right) \\
    & = - \sum_{j=1}^{n'} d'_{j} \left( \sum_{x \in \boldsymbol{D}} \sum_{i=1}^{n_x} d_{x,i} \cdot ( \dim(\mathscr{L}_{j} \cap \mathscr{L}_{x,i} / \mathscr{L}_{x,i+1}) - \dim(\mathscr{L}_{j+1} \cap \mathscr{L}_{x,i} / \mathscr{L}_{x,i+1}) )\right) \\
    & = - \sum_{j=1}^{n'} (d'_j - d'_{j-1}) \left( \sum_{x \in \boldsymbol{D}} \sum_{i=1}^{n_x} d_{x,i} \cdot \dim(\mathscr{L}_{j} \cap \mathscr{L}_{x,i} / \mathscr{L}_{x,i+1}) \right) \\
    & = - d \sum_{j=1}^{n'}  (d'_j - d'_{j-1}) \cdot \deg^{\rm loc}(\mathscr{L}_{j,\bullet}),
\end{align*}
where $d'_0=0$. In conclusion,
\begin{equation}\label{eq_mu_chi_sub}
    \langle \mu, \chi_{\boldsymbol\theta} \rangle =  - d \sum_{j=1}^{n'}  (d'_j - d'_{j-1}) \cdot \deg^{\rm loc}(\mathscr{L}_{j,\bullet}).
\end{equation}

%Since
%\begin{align*}
%\deg^{\rm loc}(\mathscr{L}_\bullet) = \sum_{x \in \boldsymbol{D}} \sum_{i=1}^{n_x} \theta_{x,i} \lambda_{x,i} = 0,
%\end{align*}
%we have
%\begin{align*}
%    \langle \mu, \chi_{\boldsymbol\theta} \rangle & = \langle \mu, \chi_{\boldsymbol\theta} \rangle + N d \sum_{x \in \boldsymbol{D}} \sum_{i=1}^{n_x} \theta_{x,i} \lambda_{x,i} \\
%    & = - d \sum_{j}\left( (d'_j - d'_{j-1}) \cdot \deg^{\rm loc}(\mathscr{L}_{j,\bullet}) \right) + N d \sum_{x \in \boldsymbol{D}} \sum_{i=1}^{n_x} \theta_{x,i} \lambda_{x,i} \\
%    & = - d \sum_{j} (d'_{j} - d'_{j-1} - N ) \cdot \deg^{\rm loc}(\mathscr{L}_{j,\bullet}) ,
%\end{align*}
%where $N$ is an arbitrary integer. Therefore, the integers $d'_j - d'_{j-1}$ can be supposed to be positive for all $j$ by choosing an appropriate integer $N$.

\hfill{\space}

We suppose that $\mathscr{L}_\bullet$ is semistable, which gives rise to $\deg^{\rm loc}(\mathscr{L}_{j,\bullet}) \leq 0$ by Lemma \ref{lem_par_deg_zero}. Then we have
\begin{align*}
    \langle \mu, \chi_{\boldsymbol\theta} \rangle = - d \sum_{j} (d'_{j} - d'_{j-1} ) \cdot \deg^{\rm loc}(\mathscr{L}_{j,\bullet}) \geq 0
\end{align*}
because the integers $d'_j$ are in decreasing order. By Proposition \ref{prop_king2.5}, we know that the corresponding point $\phi \in \mathscr{R}(\mathscr{Q}^{\boldsymbol{D}},\mathscr{I}^{\boldsymbol{D}},[\boldsymbol{P}])$ is $\chi_{\boldsymbol\theta}$-semistable.

On the other hand, we assume that $\phi$ is $\chi_{\boldsymbol\theta}$-semistable. By Proposition \ref{prop_king2.5}, the condition $\chi_{\boldsymbol\theta}(\Delta) =\{1\}$ implies that the degree of $\mathscr{L}_\bullet$ is zero. Let $\mathscr{L}'_\bullet \subseteq \mathscr{L}_\bullet$ be a filtered sub-local system. Denote by $P$ the parabolic subgroup determined by the filtration
\begin{align*}
    \mathscr{L} = \mathscr{L}_1  \supsetneq \mathscr{L}' = \mathscr{L}_2 \supsetneq \mathscr{L}_3 = \{0\}
\end{align*}
as in Example \ref{exmp_sub_local_system}. In this case, $n'=2$. We choose a cocharacter $\mu_{v_0}$ and denote by $\mu: \mathbb{G}_m \rightarrow G_{\boldsymbol{P}}$ the cocharacter determined by $\mu_{v_0}$ and $\phi$ as given in Lemma \ref{lem_cochar_equal} such that $P_\mu = P$ and $d'_2-d'_1$ is a positive integer. By Lemma \ref{lem_cochar_equal}, the limit $\lim_{t \rightarrow 0} \mu(t) \cdot \phi$ exists. Since $\phi$ is $\chi_{\boldsymbol\theta}$-semistable, we have
\begin{align*}
    \langle \mu , \chi_{\boldsymbol\theta} \rangle = -d (d'_2 - d'_1) \cdot \deg^{\rm loc}(\mathscr{L}'_{\bullet}) \geq 0
\end{align*}
by Proposition \ref{prop_king2.5}, and then,
\begin{align*}
    \deg^{\rm loc}(\mathscr{L}'_{\bullet}) \leq 0.
\end{align*}
Therefore, $\mathscr{L}_\bullet$ is semistable. The proof of the stable case is the same.
\end{proof}

\begin{rem}\label{rem_many_quiv}
We can construct a quiver $\mathscr{Q}''=(\mathscr{Q}''_0,\mathscr{Q}''_1)$ as we did in \S\ref{subsect_rep_fund_grp}, where $\mathscr{Q}''_0$ consists of a single point $v_0$ and the arrow set $\mathscr{Q}''_1$ is
\begin{align*}
    \mathscr{Q}''_1:=\{a_i, b_j, c_x : v_0 \rightarrow v_0, 1 \leq i, j \leq g, x \in \boldsymbol{D}\}.
\end{align*}
Define the relation
\begin{align*}
    \mathscr{I}''= (\prod_{i=1}^g [a_i,b_i] )( \prod_{x \in \boldsymbol{D}} c_x ) -v_0.
\end{align*}
It is easy to prove that
\begin{align*}
    \mathscr{R}(\mathscr{Q}'',\mathscr{I}'',n) \cong {\rm Hom}(\pi_1(X_{\boldsymbol{D}}), {\rm GL}_n(k)).
\end{align*}
Thus, the result of Proposition \ref{prop_para_and_orbit} also holds for special representations of the quiver $\mathscr{Q}''$. However, if we choose this quiver, it is impossible to prove the equivalence of stability conditions (Proposition \ref{prop_stab_local_sys_and_quiv_rep}). The reason is that there is only one vertex $v_0$ and the group $G$ acting on the quiver representation spaces $\mathscr{R}(\mathscr{Q}'',\mathscr{I}'',n)$ could only be a subgroup of ${\rm GL}_n(k)$, and a character $\chi: G \rightarrow \mathbb{G}_m$ cannot include all data (parabolic structures) on distinct punctures $x \in \boldsymbol{D}$.
\end{rem}

\begin{rem}\label{rem_real_weight}
In this paper, we only consider rational weights for filtered local systems for the convenience of applying King's quiver GIT techniques \cite[\S 2]{King94}. To deal with real weights, a method of approximation is necessary. Let $\boldsymbol{\theta}$ be a given collection of real weights, then one can approximate $\boldsymbol{\theta}$ by a collection of ``nearby'' rational weights $\boldsymbol{\theta}'$ such that $P_{\theta_x}=P_{\theta'_x}$ for each $x \in \boldsymbol{D}$ and the stability conditions for stable objects are equivalent. Such a method of approximation is common in dealing with parabolic/parahoric objects, see for example, \cite{BBP17} and \cite[\S8]{IS08}.

Now we briefly explain the word ``nearby''. First, 
we need the property that $P_{\theta_x}=P_{\theta_x'}$, which can be ensured if $\boldsymbol{\theta}'$ is chosen within the apartment in the Bruhat--Tits building of $\boldsymbol{\theta}$. Second, the weights $\boldsymbol{\theta}'$ are chosen to satisfy the condition that for any potential sub-local system  $\mathscr{L}'$, the sign of $\sum_{x \in \boldsymbol{D}} \sum_{i=1}^{n_x} \theta_{x,i} \cdot \dim( \mathscr{L}' \cap \mathscr{L}_{x,i} / \mathscr{L}_{x,i+1}  )$
is the same as that of $\sum_{x \in \boldsymbol{D}} \sum_{i=1}^{n_x} \theta'_{x,i} \cdot \dim( \mathscr{L}' \cap \mathscr{L}_{x,i} / \mathscr{L}_{x,i+1}  )$. Since the stability is defined by a finite set of inequalities for a given rank, such rational weights $\boldsymbol{\theta}'$  can always be found in a sufficiently small neighborhood of $\boldsymbol{\theta}$.
\end{rem}

\subsection{Moduli Space of Filtered Local Systems}

\begin{defn}
Let $\mathscr{L}_\bullet$ be a semistable filtered local system of degree zero. A \emph{Jordan--H\"older filtration} of $\mathscr{L}_\bullet$ is a filtration of filtered local systems
\begin{align*}
    \mathscr{L}_\bullet = \mathscr{L}_{1,\bullet} \supsetneq \dots \supsetneq \mathscr{L}_{n'+1,\bullet} = \{0\}
\end{align*}
such that the factors $\mathscr{L}_{j,\bullet}/\mathscr{L}_{j+1,\bullet}$ are stable filtered local systems of degree zero. We define ${\rm gr}(\mathscr{L}_\bullet):=\bigoplus_{j=1}^{n'} \mathscr{L}_{j,\bullet}/\mathscr{L}_{j+1,\bullet}$. Two filtered local systems $\mathscr{L}_\bullet$ and $\mathscr{L}'_\bullet$ are called \emph{$S$-equivalent} if ${\rm gr}(\mathscr{L}_\bullet) \cong {\rm gr}(\mathscr{L}'_\bullet)$. 
\end{defn}

Following the approach in \cite[\S 1.5]{HL97}, Jordan--H\"older filtrations for semistable filtered local systems always exist. Similar to \cite[Proposition 4.2]{King94}, we have the following result for $S$-equivalence classes

\begin{lem}\label{lem_S_equiv_and_GIT_equiv}
Let $\phi = (\phi_a)_{a \in \mathscr{Q}^{\boldsymbol{D}}_1}$ and $\phi' = (\phi'_a)_{a \in \mathscr{Q}^{\boldsymbol{D}}_1}$ be two $\chi_{\boldsymbol\theta}$-semistable points in $\mathscr{R}(\mathscr{Q}^{\boldsymbol{D}},\mathscr{I}^{\boldsymbol{D}},[\boldsymbol{P}])$. Let $\mathscr{L}_\bullet$ and $\mathscr{L}'_\bullet$ be the corresponding semistable filtered local systems of degree zero with weights $\boldsymbol\theta$ respectively. Then $\mathscr{L}_\bullet$ and $\mathscr{L}'_\bullet$ are $S$-equivalent if and only if $\phi \sim \phi'$, i.e. $\phi$ and $\phi'$ are GIT equivalent.
\end{lem}

\begin{proof}
We first fix some notations. Let $P$ (resp. $P'$) be the automorphism group of the filtration
\begin{align*}
    & \mathscr{L}= \mathscr{L}_{1} \supsetneq \dots \supsetneq \mathscr{L}_{n'+1} = \{0\} \\
    (\text{resp. } & \mathscr{L}'= \mathscr{L}'_{1} \supsetneq \dots \supsetneq \mathscr{L}'_{n'+1} = \{0\})
\end{align*}
with Levi decomposition $P=LU$ (resp. $P'=L'U'$). We write $L$ (resp. $L'$) as a product of general linear groups $\prod_{j=1}^{n'} L_j$ (resp. $\prod_{j=1}^{n'} L'_j$). With a similar calculation as in Example \ref{exmp_sub_local_system}, we have
\begin{align*}
    \deg^{\rm loc} (\mathscr{L}_{j,\bullet}/\mathscr{L}_{j+1,\bullet}) = \sum_{x \in \boldsymbol{D}} \sum_{i=1}^{n_x}  d_{x,i} \cdot \dim(\mathscr{L}_{j} \cap \mathscr{L}_{x,i} / (\mathscr{L}_{j+1} \cap \mathscr{L}_{x,i} + \mathscr{L}_{j} \cap \mathscr{L}_{x,i+1}) ).
\end{align*}
We take a cocharacter $\mu: \mathbb{G}_m \rightarrow G_{\boldsymbol{P}}$ such that the limit $\lim_{t \rightarrow 0} \mu(t) \cdot \phi$ exists and $P_\mu =P$. We introduce the following notations:
\begin{align*}
    \phi_{a_i}|_{L}:=\lim_{t \rightarrow 0} \mu_{v_0}(t) \phi_{a_i} \mu^{-1}_{v_0}(t), \ \ & \phi_{b_j}|_{L}:=\lim_{t \rightarrow 0} \mu_{v_0}(t) \phi_{b_j} \mu^{-1}_{v_0}(t), \\ \phi_{c_{x,1}}|_{L}:=\lim_{t \rightarrow 0} \mu_{v_x}(t) \phi_{c_{x,1}} \mu^{-1}_{v_0}(t), \ \ & \phi_{c_{x,2}}|_{L}:=\lim_{t \rightarrow 0} \mu_{v_0}(t) \phi_{c_{x,2}} \mu^{-1}_{v_x}(t),
\end{align*}
and then
\begin{align*}
    \lim_{t \rightarrow 0} \mu(t) \cdot \phi = (\phi_a|_{L})_{a \in \mathscr{Q}^{\boldsymbol{D}}_1}.
\end{align*}
Clearly, we have $\phi_{a_i}|_L$, $\phi_{b_j}|_{L}$, $(\phi_{c_{x,2}}|_{L})(\phi_{c_{x,1}}|_{L}) \in L$. Moreover, for each $x \in \boldsymbol{D}$, we choose $g_x \in {\rm GL}_n(k)$ such that $g_x \phi_{c_{x,2}} \in P_x$, $\phi_{c_{x,1}} g^{-1}_x \in L_x$ and $\mu_{v_x}(t) = g_x \mu_{v_0}(t) g^{-1}_x$ by Lemma \ref{lem_cochar_equal}. The same notations will be applied for $\phi'$, and we will not repeat here.

\hfill{\space}

We first suppose that $\mathscr{L}_\bullet$ and $\mathscr{L}'_\bullet$ are $S$-equivalent. We take a cocharacter $\mu: \mathbb{G}_m \rightarrow G_{\boldsymbol{P}}$ (resp. $\mu': \mathbb{G}_m \rightarrow G_{\boldsymbol{P}}$) such that the limit $\lim_{t \rightarrow 0} \mu(t) \cdot \phi$ (resp. $\lim_{t \rightarrow 0} \mu'(t) \cdot \phi'$) exists and $P_\mu =P$ (resp. $P_{\mu'} = P'$). Since $\deg^{\rm loc} (\mathscr{L}_{j,\bullet}/\mathscr{L}_{j+1,\bullet}) = 0$ and $\deg^{\rm loc} (\mathscr{L}'_{j,\bullet}/\mathscr{L}'_{j+1,\bullet}) = 0$ by assumption, we have
\begin{align*}
    \deg^{\rm loc}(\mathscr{L}_{j, \bullet}) = \deg^{\rm loc}(\mathscr{L}'_{j, \bullet}) = 0.
\end{align*}
Therefore,
\begin{align*}
    \langle \mu, \chi_{\boldsymbol\theta} \rangle = \langle \mu', \chi_{\boldsymbol\theta} \rangle = 0
\end{align*}
by formula \ref{eq_mu_chi_sub}.

Denote by $\rho$ and $\rho'$ the corresponding representation of $\mathscr{L}_\bullet$ and $\mathscr{L}'_\bullet$ respectively. The isomorphism ${\rm gr}(\mathscr{L}_\bullet) \cong {\rm gr}(\mathscr{L}'_\bullet)$ implies that there exists $g \in {\rm GL}_n(k)$ such that $g P g^{-1} = P'$ and $g (\rho|_{L}) g^{-1}  = \rho'|_{L'}$, where $\rho|_{L}$ (resp. $\rho|_{L'}$) is the restriction of the representation $\rho$ (resp. $\rho'$) to $L$ (resp. $L'$). Thus, 
\begin{align*}
    g (\rho|_{L})(a_i) g^{-1} = (\rho'|_{L'})(a_i), \ \ g (\rho|_{L})(b_j) g^{-1} = (\rho'|_{L'})(b_j), \ \ g (\rho|_{L})(c_x) g^{-1} = (\rho'|_{L'})(c_x),
\end{align*}
and then
\begin{align*}
    g (\phi_{a_i}|_{L}) g^{-1} = \phi'_{a_i}|_{L'}, \ \ g (\phi_{b_j}|_{L}) g^{-1} = \phi'_{b_j}|_{L'}, \ \ g (\phi_{c_{x,2}}|_{L})(\phi_{c_{x,1}}|_{L}) g^{-1} = (\phi'_{c_{x,2}}|_{L'})( \phi'_{c_{x,1}}|_{L'}).
\end{align*}
Now we consider $g_x (\phi_{c_{x,2}}|_L)$. The reformulation
\begin{align*}
    g_x (\phi_{c_{x,2}}|_L) = \lim_{t \rightarrow 0} (g_x \mu_{v_0}(t) g^{-1}_x) (g_x \phi_{c_{x,2}}) \mu_{v_x}(t)^{-1}
\end{align*}
shows that $g_x (\phi_{c_{x,2}}|_L) \in L_x$ because $g_x \mu_{v_0}(t) g^{-1}_x = \mu_{v_x}(t)$ is a cocharacter of $L_x$ by Lemma \ref{lem_cochar_equal}. With a similar argument, we have
\begin{align*}
    g'_x (\phi'_{c_{x,2}}|_{L'}) \in L_x, \ \ (\phi_{c_{x,1}}|_L) g^{-1}_x \in L_x, \ \ (\phi'_{c_{x,1}}|_{L'}) (g'_x)^{-1} \in L_x.
\end{align*}
Therefore,
\begin{align*}
    & g (\phi_{c_{x,2}}|_{L})(\phi_{c_{x,1}}|_{L}) g^{-1} = gg^{-1}_x (g_x \phi_{c_{x,2}}|_{L}) (\phi_{c_{x,1}}|_{L} g^{-1}_x ) g_x g^{-1} \\
    = & (g'_x)^{-1} (g'_x \phi'_{c_{x,2}}|_{L'}) (\phi'_{c_{x,1}}|_{L'} (g'_x)^{-1} ) g'_x = (\phi'_{c_{x,2}}|_{L'}) (\phi'_{c_{x,1}}|_{L'}).
\end{align*}
Note that the four elements
\begin{align*}
    g_x \phi_{c_{x,2}}|_{L}, \ \ \phi_{c_{x,1}}|_{L} g^{-1}_x, \ \ g'_x \phi'_{c_{x,2}}|_{L'}, \ \
    \phi'_{c_{x,1}}|_{L'} (g'_x)^{-1}
\end{align*}
are in $L_x$, and then $g'_x g g^{-1}_x \in L_x$ because the normalizer of $L_x$ is itself. We define
\begin{align*}
    & g_{v_0}:=g, \\
    & g_{v_x} := (\phi'_{c_{x,1}}|_{L'}) g (\phi_{c_{x,1}}|_{L})^{-1}.
\end{align*}
Since 
\begin{align*}
    g_{v_x} = (\phi'_{c_{x,1}}|_{L'} (g'_x)^{-1}) (g'_x g g^{-1}_x) (g_x(\phi_{c_{x,1}}|_{L})^{-1}),
\end{align*}
we have $(g_v)_{v \in \mathscr{Q}^{\boldsymbol{D}}_0} \in G_{\boldsymbol{P}}$. It is easy to check that 
\begin{align*}
    (\phi'_a|_{L'})_{a \in \mathscr{Q}^{\boldsymbol{D}}_1} = (g_v)_{v \in \mathscr{Q}^{\boldsymbol{D}}_0} \cdot (\phi_a|_{L})_{a \in \mathscr{Q}^{\boldsymbol{D}}_1}.
\end{align*}
This gives rise to $\phi \sim \phi'$ from Proposition \ref{prop_king2.6}.

\hfill{\space}

For the other direction, we suppose that $\phi \sim \phi'$. By Proposition \ref{prop_king2.6}, there exist cocharacters $\mu$ and $\mu'$ such that $\langle \mu, \chi_{\boldsymbol\theta} \rangle = \langle \mu', \chi_{\boldsymbol\theta} \rangle = 0$ and the limits $\lim_{t \rightarrow 0} \mu(t) \cdot \phi$ and $\lim_{t \rightarrow 0} \mu'(t) \cdot \phi'$ are in the same closed $G_{\boldsymbol{P}}$-orbit. Following the notations
\begin{align*}
    \lim_{t \rightarrow 0} \mu(t) \cdot \phi = (\phi_a|_{L})_{a \in \mathscr{Q}^{\boldsymbol{D}}_1}, \ \ \lim_{t \rightarrow 0} \mu'(t) \cdot \phi' = (\phi'_a|_{L'})_{a \in \mathscr{Q}^{\boldsymbol{D}}_1} ,
\end{align*}
we suppose that 
\begin{align*}
    (g_v)_{v \in \mathscr{Q}^{\boldsymbol{D}}_0} \cdot (\phi_a|_{L})_{a \in \mathscr{Q}^{\boldsymbol{D}}_1} = (\phi'_a|_{L'})_{a \in \mathscr{Q}^{\boldsymbol{D}}_1},
\end{align*}
where $(g_v)_{v \in \mathscr{Q}^{\boldsymbol{D}}_0} \in G_{\boldsymbol{P}}$. By Proposition \ref{prop_para_and_orbit}, we have $g_{v_0} (\rho|_{L}) g_{v_0}^{-1} = \rho'|_{L'}$. Therefore, ${\rm gr}(\mathscr{L}_\bullet) \cong {\rm gr}(\mathscr{L}'_\bullet)$. 

By formula \ref{eq_mu_chi_sub}, we have
\begin{align*}
    \langle \mu, \chi_{\boldsymbol\theta} \rangle =  - d \sum_{j=1}^{n'}  (d'_j - d'_{j-1}) \cdot \deg^{\rm loc}(\mathscr{L}_{j,\bullet}).
\end{align*}
Since $\langle \mu, \chi_{\boldsymbol\theta} \rangle = 0$, we get $\deg^{\rm loc}(\mathscr{L}_{j,\bullet}) = 0$. Therefore, $\deg^{\rm loc} (\mathscr{L}_{j,\bullet}/\mathscr{L}_{j+1,\bullet}) = 0$. The same argument holds for $\deg^{\rm loc} (\mathscr{L}'_{j,\bullet}/\mathscr{L}'_{j+1,\bullet}) = 0$. 

At the end, we can take maximal filtrations $\mathscr{L}_{j,\bullet}$ 
 and $\mathscr{L}'_{j,\bullet}$ of $\mathscr{L}_{\bullet}$ and $\mathscr{L}'_{\bullet}$ respectively and then the quotients $\mathscr{L}_{j,\bullet}/\mathscr{L}_{j+1 , \bullet}$ and $\mathscr{L}'_{j,\bullet}/\mathscr{L}'_{j+1,\bullet}$ are stable with degree zero. Therefore, arguments above show that $\mathscr{L}_\bullet$ and $\mathscr{L}'_\bullet$ are $S$-equivalent.

\end{proof}

We consider the moduli functor
\begin{align*}
    \widetilde{\mathcal{M}}_{\rm B}(X_{\boldsymbol{D}}, [\boldsymbol{P}],\boldsymbol\theta) : ({\rm Sch}/k)^{\rm op} \rightarrow {\rm Set},
\end{align*}
where ${\rm Sch}/k$ is the category of $k$-schemes, defined by sending each $k$-scheme $S$ to the set of isomorphism classes of $S$-flat families of semistable filtered local systems of type $[\boldsymbol{P}]$ with weights $\boldsymbol\theta$ on $X_{\boldsymbol{D}}$. 

\begin{thm}\label{thm_tame_Betti_lc}
Define
\begin{align*}
    \mathcal{M}_{\rm B}(X_{\boldsymbol{D}},[\boldsymbol{P}],\boldsymbol\theta) := \mathscr{R}(\mathscr{Q}^{\boldsymbol{D}},\mathscr{I}^{\boldsymbol{D}},[\boldsymbol{P}]) /\!\!/ (G_{\boldsymbol{P}} , \chi_{\boldsymbol\theta} )
\end{align*}
to be the GIT quotient with respect to $\chi_{\boldsymbol\theta}$. Then,
\begin{enumerate}
    \item $\mathcal{M}_{\rm B}(X_{\boldsymbol{D}},[\boldsymbol{P}],\boldsymbol\theta)$ is a quasi-projective variety;
    \item $\mathcal{M}_{\rm B}(X_{\boldsymbol{D}},[\boldsymbol{P}],\boldsymbol\theta)$ universally co-represents the moduli functor $\widetilde{\mathcal{M}}_{\rm B}(X_{\boldsymbol{D}},[\boldsymbol{P}],\boldsymbol\theta)$;
    \item the $k$-points are in one-to-one correspondence with $S$-equivalence classes of degree zero semistable filtered local systems of type $[\boldsymbol{P}]$ with weights $\boldsymbol\theta$;
    \item there exists an open subset $\mathcal{M}^s_{\rm B}(X_{\boldsymbol{D}},[\boldsymbol{P}],\boldsymbol\theta)$, whose $k$-points correspond to isomorphism classes of degree zero stable filtered local systems of type $[\boldsymbol{P}]$ with weights $\boldsymbol\theta$.
\end{enumerate}
\end{thm}

\begin{proof}
Since $\mathscr{R}(\mathscr{Q}^{\boldsymbol{D}},\mathscr{I}^{\boldsymbol{D}}, [\boldsymbol{P}])$ is quasi-projective, the first statement follows directly. By Proposition \ref{prop_stab_local_sys_and_quiv_rep}, $\chi_{\boldsymbol\theta}$-semistable points are in one-to-one correspondence with semistable local systems of type $[\boldsymbol{P}]$ with weights $\boldsymbol\theta$. Moreover, two semistable points are GIT equivalent if and only if the corresponding filtered local systems are $S$-equivalent (see Lemma \ref{lem_S_equiv_and_GIT_equiv}). Therefore, with the same argument as in \cite[\S 6]{Sim94b}, the theorem follows directly from the geometric invariant theory (see \cite{MFK94,HL97,Sim94a} for instance). 
\end{proof}

\begin{exmp}\label{exmp_trivial_local}
In this example, we discuss the case of trivial weights. More precisely, $\theta_x=0$ for $x \in \boldsymbol{D}$ and $P_x={\rm GL}_n(k)$. The triviality of $P_x$ gives $L_{x} = {\rm GL}_n(k)$. Thus,  
\begin{align*}
G_{\boldsymbol{P}}=\mathrm{GL}_n(k)\times\prod_{x\in\boldsymbol{D}}\mathrm{GL}_n(k)
\end{align*}
and
\begin{align*}
    \mathscr{R}(\mathscr{Q}^{\boldsymbol{D}},\mathscr{I}^{\boldsymbol{D}},[\boldsymbol{P}])= \mathscr{R}(\mathscr{Q}^{\boldsymbol{D}},\mathscr{I}^{\boldsymbol{D}},n).
\end{align*}
The triviality of each $\theta_x$ shows that the character 
\begin{align*}
    \chi_{\boldsymbol{\theta}}: G_{\boldsymbol{P}}\to\mathbb{G}_m
\end{align*}
is also trivial in this case. Thus, any point $(\phi_a)_{a \in \mathscr{Q}^{\boldsymbol{D}}_1} \in \mathscr{R}(\mathscr{Q}^{\boldsymbol{D}},\mathscr{I}^{\boldsymbol{D}},n)$ is $\chi_{\boldsymbol{\theta}}$-semistable automatically, and the moduli space 
\begin{align*}
    \mathcal{M}_{\rm B}(X_{\boldsymbol{D}},[\boldsymbol{P}],\boldsymbol\theta)=\mathscr{R}(\mathscr{Q}^{\boldsymbol{D}},\mathscr{I}^{\boldsymbol{D}},n)/\!\!/G_{\boldsymbol{P}}\cong\mathrm{Hom}(\pi_1(X_{\boldsymbol{D}}),\mathrm{GL}_n(k))/\!\!/\mathrm{GL}_n(k)
\end{align*}
is affine, whose points correspond to semisimple fundamental group representations. The latter affine GIT quotient is usually called the \emph{character variety}.
\end{exmp}

\begin{exmp}
In this example, we construct a stable filtered $G$-local system, which does not correspond to an irreducible or semisimple representation. Therefore, the Betti moduli spaces in the tame NAHC are not isomorphic to the character varieties in general. 

Let $X= \mathbb{P}^1$, and $\boldsymbol{D}$ is a collection of three distinct points $x_1,x_2,x_3$. Let $\mathscr{L}$ be a rank two vector space. Without loss of generality, we set $\mathscr{L}=\mathbb{C}^2$ and fix the standard basis $\{e_1,e_2\}$ of $\mathbb{C}^2$. Now we define a filtered local system $\mathscr{L}_\bullet$ as follows. For $x_1$ and $x_2$, the filtration $\mathscr{L}_{x_i,\bullet}$ is given as follows
\begin{align*}
    \mathscr{L} = \mathscr{L}_{x_i,1} \supsetneq \mathscr{L}_{x_i,2} = \langle e_1 \rangle \supsetneq \{0\}, \quad i=1,2,
\end{align*}
with weights
\begin{align*}
    \theta_{x_i,1} = \frac{1}{3}, \quad \theta_{x_i,2} = -\frac{1}{3}, \quad i=1,2,
\end{align*}
and the representation $\rho \in {\rm Hom}(\pi_1(X_{\boldsymbol{D}}), {\rm GL}_2(\mathbb{C}))$ is defined as
\begin{align*}
    \rho(c_{x_1}) = \begin{pmatrix}
        1 & 1 \\ 0 & 1
    \end{pmatrix}, \quad \rho(c_{x_2}) = \begin{pmatrix}
        1 & -1 \\ 0 & 1
    \end{pmatrix}.
\end{align*}
For $x_3$, we equip it with a trivial parabolic structure. In other words, the filtration $\mathscr{L}_{x_3, \bullet}$ is trivial with trivial weights. Moreover, $\rho(c_{x_3})$ is the identity matrix. Clearly, the filtered local system $\mathscr{L}_\bullet$ is of degree zero, i.e.
\begin{align*}
    \deg^{\rm loc}(\mathscr{L}_\bullet) = \sum_{j=1}^3 \sum_{i=1}^2 \theta_{x_j,i} = 0.
\end{align*}

From the construction of $\mathscr{L}_\bullet$, it is easy to find that it has a unique rank one sub-filtered local system $\mathscr{L}'_\bullet \subseteq \mathscr{L}_\bullet$. This sub-filtered local system $\mathscr{L}'_\bullet$ is determined by the subspace $\mathscr{L}' = \langle e_1 \rangle \subseteq \mathscr{L}$ because $\mathscr{L}'$ is the unique subspace of dimension one preserved by the representation $\rho$. Furthermore, the parabolic structure of $\mathscr{L}'_\bullet$ at each puncture $x_i$ is given as follows
\begin{align*}
    \mathscr{L}' = \mathscr{L}'_{x_i,1} \supsetneq 0
\end{align*}
with $\theta'_{x_i,1} = -\frac{1}{3}$ for $i=1,2$ and $\theta'_{x_3,1} = 0$. Clearly, 
\begin{align*}
    \deg^{\rm loc} (\mathscr{L}'_\bullet) = - \frac{2}{3} < 0 .
\end{align*}
Therefore, $\mathscr{L}_\bullet$ is stable. However, the corresponding representation $\rho$ is not irreducible because $\rho(c_{x_1})$ and $\rho(c_{x_2})$ are upper triangular matrices and it is only an indecomposable representation. Therefore, this stable $\boldsymbol\theta$-filtered local system $\mathscr{L}_\bullet$ does not correspond to a point in the character variety.
\end{exmp}

\section{Moduli Space of Filtered G-local Systems}\label{sect_fgls}

In this section, we generalize the result for filtered local systems in \S\ref{sect_fls} to filtered $G$-local systems. The stability condition for filtered $G$-local systems is given in Definition \ref{defn_stab_cond_G}, which is an analogue of Ramanathan's stability condition on principal bundles and is called the \emph{$R$-stability condition} in this paper. Following Construction \ref{cons_new_quiv} we construct a new space of quiver representations $\mathscr{R}(\mathscr{Q}^{\boldsymbol{D}},\mathscr{I}^{\boldsymbol{D}},\boldsymbol\theta)$, and show that the isomorphism classes of filtered $G$-local systems on $X_{\boldsymbol{D}}$ are in one-to-one correspondence with orbits in the quiver representation space (Proposition \ref{prop_para_and_orbit_G}). Following the equivalence of stability conditions (Proposition \ref{prop_stab_local_sys_and_quiv_rep_G}), we construct the moduli space of filtered $G$-local systems (Theorem \ref{thm_tame_Betti_lc_G}).

\subsection{Filtered G-local Systems}
Let $G$ be a connected reductive group over $k$ with a fixed maximal torus $T$. Denote by $\mathcal{R}$ the set of roots. There is a natural pairing of cocharacters and characters
\begin{align*}
    \langle \cdot , \cdot \rangle : {\rm Hom}(\mathbb{G}_m ,T) \times {\rm Hom}(T, \mathbb{G}_m) \rightarrow \mathbb{Z}.
\end{align*}
This pairing can be easily extended to be with coefficients in $\mathbb{Q}$ and $\mathbb{R}$. Before we move on, we consider the following case about the natural pairing. Let $P$ be a parabolic subgroup of $G$ with a given maximal torus $T_P$. Note that $T_P$ may not be the same as $T$. Let $\chi$ be a character of $P$ and let $\mu$ be a cocharacter of $T$. Given $g \in G$ such that $T = g T_P g^{-1}$, we define
\begin{align*}
    (  \mu , \chi ) := \langle g^{-1} \mu g , \chi \rangle = \langle \mu , g \chi g^{-1} \rangle,  
\end{align*}
where $(g^{-1} \mu g)(t):= g^{-1} \mu(t) g$ is a cocharacter of $g^{-1} T g$ and $(g \chi g^{-1})(p):=\chi(g^{-1} p g)$ is a character of $g P g^{-1}$.

A \emph{weight $\theta$} is a cocharacter of $T$ with rational number coefficients, i.e. $\theta \in {\rm Hom}(\mathbb{G}_m ,T) \otimes_{\mathbb{Z}} \mathbb{Q}$. A weight $\theta$ determines a parabolic subgroup
\begin{align*}
    P_\theta:=\{ g \in G \, | \, \lim_{t \rightarrow 0} \theta(t) g \theta(t)^{-1} \text{ exists }  \}.
\end{align*}
Denote by $L_\theta$ the Levi subgroup of $P_\theta$. Note that $P_{-\theta}$ is the opposite parabolic subgroup with the same Levi subgroup $L_\theta$.

Consider the space of fundamental group representations ${\rm Hom}(\pi_1(X_{\boldsymbol{D}}),G)$, on which there is a natural $G$-action via conjugation. Let $\rho: \pi_1(X_{\boldsymbol{D}}) \rightarrow G$ be a representation (a homomorphism of groups), and denote by $[\rho]$ the corresponding points in ${\rm Hom}(\pi_1(X_{\boldsymbol{D}}),G)$. A \emph{$G$-local system} on $X_{\boldsymbol{D}}$ is regarded as a representation $\rho: \pi_1(X_{\boldsymbol{D}}) \rightarrow G$. Two $G$-local systems are isomorphic if the representations are in the same $G$-orbit in ${\rm Hom}(\pi_1(X_{\boldsymbol{D}}),G)$.

\begin{defn}
Let $\boldsymbol\theta = \{\theta_x, x \in \boldsymbol{D}\}$ be a collection of weights. A \emph{$\boldsymbol\theta$-filtered $G$-local system} is a representation $\rho: \pi_1(X_{\boldsymbol{D}}) \rightarrow G$ such that $\rho(c_x)$ is conjugate to an element in $P_{\theta_x}$ for every $x \in \boldsymbol{D}$.
\end{defn}

\begin{con}\label{cons_new_quiv_G}
Given a collection of weights $\boldsymbol\theta$, we use a similar approach as Construction \ref{cons_new_quiv} to construct a quasi-projective variety $\mathscr{R}(\mathscr{Q}^{\boldsymbol{D}},\mathscr{I}^{\boldsymbol{D}}, \boldsymbol\theta)$. 

Following the same notation as in Construction \ref{cons_new_quiv}, let $\mathscr{R}(\mathscr{Q}^{\boldsymbol{D}},G)$ be the set of elements $(\phi_a)_{a \in \mathscr{Q}^{\boldsymbol{D}}_1}$, where $\phi_a \in G$ for every $a \in \mathscr{Q}^{\boldsymbol{D}}_1$. Clearly, $\mathscr{R}(\mathscr{Q}^{\boldsymbol{D}},G)$ is an affine variety. We define a $(\prod_{x \in \boldsymbol{D}} G_x)$-action on $\mathscr{R}(\mathscr{Q}^{\boldsymbol{D}},G)$ as follows
\begin{align*}
    (g_x) \cdot (\phi_{a_i},\phi_{b_j},\phi_{c_{x,1}},\phi_{c_{x,2}}) := (\phi_{a_i},\phi_{b_j},  \phi_{c_{x,1}}g_x^{-1}, g_x \phi_{c_{x,2}}),
\end{align*}
where $G_x = G$ for each $x \in \boldsymbol{D}$. Denote by $\boldsymbol{P}$ the collection of parabolic subgroups $\{P_{\theta_x}, x \in \boldsymbol{D}\}$ given by $\boldsymbol\theta$. We define a closed subset
\begin{align*}
    \mathscr{R}(\mathscr{Q}^{\boldsymbol{D}},\boldsymbol{P}) \subseteq \mathscr{R}(\mathscr{Q}^{\boldsymbol{D}},G),
\end{align*}
whose elements $\phi= (\phi_a)_{a \in \mathscr{Q}^{\boldsymbol{D}}_1}$ satisfy the condition 
\begin{align*}
    \phi_{c_{x,1}} \in L_{\theta_x}, \quad \phi_{c_{x,2}} \in P_{-\theta_x}
\end{align*}
for $x \in \boldsymbol{D}$. 

Denote by $\mathscr{R}'(\mathscr{Q}^{\boldsymbol{D}},\boldsymbol{P})$ the fiber product
\begin{center}
\begin{tikzcd}
\mathscr{R}'(\mathscr{Q}^{\boldsymbol{D}},\boldsymbol{P}) \arrow[r, dotted] \arrow[d,dotted] & \mathscr{R}(\mathscr{Q}^{\boldsymbol{D}},\boldsymbol{P}) \arrow[d] \\
(\prod_{x \in \boldsymbol{D}} G_x) \times \mathscr{R}(\mathscr{Q}^{\boldsymbol{D}},G) \arrow[r] & \mathscr{R}(\mathscr{Q}^{\boldsymbol{D}},G) \ ,
\end{tikzcd}
\end{center}
which parametrizes tuples
\begin{align*}
    ( (g_x) , (\phi_{a_i},\phi_{b_j},\phi_{c_{x,1}},\phi_{c_{x,2}}) ) 
\end{align*}
such that
\begin{align*}
    \phi_{c_{x,1}} g^{-1}_x \in L_{\theta_x}, \ \ g_x \phi_{c_{x,2}} \in P_{-\theta_x}.
\end{align*}
Then, we define
\begin{align*}
    \mathscr{R}(\mathscr{Q}^{\boldsymbol{D}},[\boldsymbol{P}]) := \mathscr{R}'(\mathscr{Q}^{\boldsymbol{D}},\boldsymbol{P})|_{\mathscr{R}(\mathscr{Q}^{\boldsymbol{D}},G)},
\end{align*}
which parametrizes tuples $(\phi_{a_i},\phi_{b_j},\phi_{c_{x,1}},\phi_{c_{x,2}})$ such that for each $x \in \boldsymbol{D}$, there exists $g_x \in G_x = G$ satisfying 
\begin{align*}
    \phi_{c_{x,1}} g^{-1}_x \in L_{\theta_x}, \ \ g_x \phi_{c_{x,2}} \in P_{-\theta_x}.
\end{align*}
By adding the relation $\mathscr{I}^{\boldsymbol{D}}$, we obtain a closed subvariety $\mathscr{R}(\mathscr{Q}^{\boldsymbol{D}} , \mathscr{I}^{\boldsymbol{D}} , [\boldsymbol{P}]) \subseteq \mathscr{R}(\mathscr{Q}^{\boldsymbol{D}}, [\boldsymbol{P}])$. Summing up, we obtain a quasi-projective variety $\mathscr{R}(\mathscr{Q}^{\boldsymbol{D}},\mathscr{I}^{\boldsymbol{D}}, [\boldsymbol{P}])$. Note that in the above construction, we can take another collection $\boldsymbol{P}'=\{P'_x, \ x \in \boldsymbol{D}\}$ such that $P'_x$ is conjugate to $P_{\theta_x}$ and define $\mathscr{R}(\mathscr{Q}^{\boldsymbol{D}}, \mathscr{I}^{\boldsymbol{D}},[\boldsymbol{P}'])$ in the same way. It is easy to check 
\begin{align*}
    \mathscr{R}(\mathscr{Q}^{\boldsymbol{D}},\mathscr{I}^{\boldsymbol{D}},[\boldsymbol{P}]) \cong \mathscr{R}(\mathscr{Q}^{\boldsymbol{D}},\mathscr{I}^{\boldsymbol{D}},[\boldsymbol{P}']). 
\end{align*}
Therefore, we would like to use the notation
\begin{align*}
    \mathscr{R}(\mathscr{Q}^{\boldsymbol{D}},\mathscr{I}^{\boldsymbol{D}},\boldsymbol\theta) := \mathscr{R}(\mathscr{Q}^{\boldsymbol{D}},\mathscr{I}^{\boldsymbol{D}},[\boldsymbol{P}]).
\end{align*}
Now we define a reductive group 
\begin{align*}
    G_{\boldsymbol\theta} = \prod_{v \in \mathscr{Q}^{\boldsymbol{D}}_0} G_v := G \times \prod_{x \in \boldsymbol{D}} L_{\theta_x},
\end{align*}
where $G_{v_0}= G$ and $G_{v_x} = L_{\theta_x}$ for $x \in \boldsymbol{D}$. The $G_{\boldsymbol{\theta}}$-action on $\mathscr{R}(\mathscr{Q}^{\boldsymbol{D}},\mathscr{I}^{\boldsymbol{D}},\boldsymbol{\theta})$ is given by
\begin{align*}
    (g_{v_0},g_{v_x}) \cdot (\phi_{a_i}, \phi_{b_j}, \phi_{c_{x,1}} , \phi_{c_{x,2}}) = (g_{v_0} \phi_{a_i} g_{v_0}^{-1}, g_{v_0} \phi_{b_j} g_{v_0}^{-1}, g_{v_x} \phi_{c_{x,1}} g_{v_0}^{-1}, g_{v_0} \phi_{c_{x,2}} g_{v_x}^{-1}).
\end{align*}
\end{con}

\begin{prop}\label{prop_para_and_orbit_G}
There exists a one-to-one correspondence between $G_{\boldsymbol\theta}$-orbits in $\mathscr{R}(\mathscr{Q}^{\boldsymbol{D}},\mathscr{I}^{\boldsymbol{D}},\boldsymbol{\theta})$ and isomorphism classes of $\boldsymbol\theta$-filtered $G$-local systems on $X_{\boldsymbol{D}}$.
\end{prop}

\begin{proof}
This proposition is an analogue of Proposition \ref{prop_para_and_orbit}, and we only give the construction of a $\boldsymbol\theta$-filtered $G$-local system from a given point $\phi$ in $\mathscr{R}(\mathscr{Q}^{\boldsymbol{D}},\mathscr{I}^{\boldsymbol{D}},\boldsymbol\theta)$. Let $(\phi_a)_{a \in \mathscr{Q}^{\boldsymbol{D}}_1} \in \mathscr{R}(\mathscr{Q}^{\boldsymbol{D}},\mathscr{I}^{\boldsymbol{D}},\boldsymbol\theta)$ be a point. Then there exists $g_x \in G$ such that
\begin{align*}
    \phi_{c_{x,1}} g^{-1}_x \in L_{\theta_x}, \quad g_x \phi_{c_{x,2}} \in P_{-\theta_x}
\end{align*}
for each $x \in \boldsymbol{D}$. We define a fundamental group representation $\rho: \pi_1(X_{\boldsymbol{D}}) \rightarrow G$, i.e. a $G$-local system, as follows
\begin{align*}
\rho(a_i) = \phi_{a_i}, \quad \rho(b_j) = \phi_{b_j}, \quad \rho(c_x) = \phi_{c_x} = \phi_{c_{x,2}} \phi_{c_{x,1}}.
\end{align*}
Clearly, 
\begin{align*}
    g_x \phi_{c_x} g_x^{-1} = g_x \phi_{c_{x,2}} \phi_{c_{x,1}} g^{-1}_x \in P_{-\theta_x} .
\end{align*}
Thus, $\rho$ is a $\boldsymbol\theta$-filtered $G$-local system.
\end{proof}

\subsection{R-stability Condition}\label{subsect_stab_R}
Let $P$ be a parabolic subgroup with a given maximal torus $T_P$. Given a $G$-local system $\rho : \pi_1(X_{\boldsymbol{D}}) \rightarrow G$, a parabolic subgroup $P$ is \emph{compatible} with $\rho$, if there is a lifting
\begin{center}
\begin{tikzcd}
& & P \arrow[d] \\
\pi_1(X_{\boldsymbol{D}}) \arrow[urr, dotted] \arrow[rr,"\rho"] & & G
\end{tikzcd}
\end{center}
In other words, the representation $\rho$ is well-defined when restricted to $P$. Given a collection of weights $\boldsymbol\theta$, suppose that $\rho$ is $\boldsymbol\theta$-filtered and compatible with $P$. Let $\chi$ be a character of $P$ and we can choose $g_x \in G$ such that $g_x \rho(c_x) g^{-1}_x \in P_{-\theta_x}$ and $T = g_x T_P g_x^{-1}$. The pairing $(\theta_x , \chi )$ in this case is defined as follows:
\begin{align*}
    (\theta_x , \chi ) := \langle g^{-1}_x \theta_x g_x , \chi \rangle = \langle \theta_x , g_x \chi g_x^{-1} \rangle.
\end{align*}
Moreover, the pairing $(\theta_x , \chi )$ does not depend on $g_x$ we choose. We briefly explain the reason as follows. Suppose that both $g_x$ and $g'_x$ satisfy the above conditions, i.e.
\begin{align*}
    T = g_x T_P g_x^{-1} = g'_x T_P (g'_x)^{-1}, \quad g_x \rho(c_x) g^{-1}_x, g'_x \rho(c_x) (g'_x)^{-1} \in P_{-\theta_x}.
\end{align*}
Since $g_x \rho(c_x) g^{-1}_x, g'_x \rho(c_x) (g'_x)^{-1} \in P_{-\theta_x}$ are conjugate, there exists $p_x \in P_{-\theta_x}$ such that 
\begin{align*}
    p_x (g_x \rho(c_x) g^{-1}_x) p_x^{-1}  = g'_x \rho(c_x) (g'_x)^{-1}.
\end{align*}
Furthermore, $(g'_x)^{-1} p_x g_x \in P$ because $\rho(c_x) \in P$ and the normalizer of $P$ is itself. Summing up, we have
\begin{align*}
    p_x \in P_{-\theta_x}, \quad (g'_x)^{-1} p_x g_x \in P.
\end{align*}
Note that $p_x \in P_{-\theta_x}$ and by Levi decomposition, we have $p_x = l_x u_x$, where $l_x$ is the Levi part and $u_x$ is the unipotent part. Then,
\begin{align*}
    \langle p_x^{-1} \theta_x p_x, g_x \chi g_x^{-1} \rangle = \langle u_x^{-1} l_x^{-1} \theta_x l_x u_x, g_x \chi g_x^{-1} \rangle = \langle u_x^{-1} \theta_x u_x, g_x \chi g_x^{-1} \rangle = \langle \theta_x, g_x \chi g_x^{-1} \rangle,
\end{align*}
where the second equality holds because $l_x$ commutes with $\theta_x$ and the third equality holds because the pairing does not depend on the unipotent part. Since characters are preserved under conjugation, we have
\begin{align*}
    \chi = ((g'_x)^{-1} p_x g_x) \chi (g_x^{-1} p_x^{-1} g'_x).
\end{align*}
Therefore,
\begin{align*}
    \langle \theta_x , g_x \chi g_x^{-1} \rangle = \langle \theta_x , p_x g_x \chi g_x^{-1} p_x^{-1} \rangle = \langle \theta_x , g'_x \chi (g'_x)^{-1} \rangle.
\end{align*}
This actually shows that the pairing $(\theta_x,\chi)$ does not depend on $g_x$ we choose.

We follow Ramanathan's definition of the stability condition on principal bundles \cite{Ram75,Ram96a} to define the stability condition of filtered $G$-local systems, which is called the \emph{$R$-stability condition}. Moreover, we refer the reader to \cite[\S 2]{Ram75} for the definition of \emph{dominant} (resp. \emph{anti-dominant}) characters. When $G={\rm GL}_n(k)$ and the degree is zero, it is equivalent to Definition \ref{defn_fls_stab}.

\begin{defn}\label{defn_stab_cond_G}
A $\boldsymbol\theta$-filtered $G$-local system $\rho$ is \emph{$R$-semistable} (resp. \emph{$R$-stable}), if for
\begin{itemize}
    \item any proper parabolic subgroup $P$ compatible with $\rho$;
    \item any nontrivial anti-dominant character $\chi$ of $P$, which is trivial on the center of $P$,
\end{itemize}
we have
\begin{align*}
\deg^{\rm loc} \rho(P,\chi):= (\boldsymbol\theta,\chi ) = \sum_{x \in \boldsymbol{D}} ( \theta_x, \chi ) \geq 0 \quad (\text{resp. } > 0) .
\end{align*}
\end{defn}

\begin{defn}
	A $\boldsymbol\theta$-filtered $G$-local system $\rho$ is of \emph{degree zero}, if for any character $\chi$ of $G$, we have $( \boldsymbol\theta , \chi ) = 0$. Note that when $G$ is semisimple, this condition is trivial.
\end{defn}

When we discuss stability conditions, both Proposition \ref{prop_king2.5} and Definition \ref{defn_stab_cond_G} require that characters act trivially on the center of $G$. Therefore, everything in this subsection is given under the assumption that $G$ is semisimple apart from the statement of Proposition \ref{prop_stab_local_sys_and_quiv_rep_G}.

As a free $\mathbb{Q}$-module, let $\{e_i\}$ (resp. $\{e^*_i\}$) be a basis of ${\rm Hom}(\mathbb{G}_m,T) \otimes_{\mathbb{Z}} \mathbb{Q}$ (resp. ${\rm Hom}(T, \mathbb{G}_m) \otimes_{\mathbb{Z}} \mathbb{Q}$) such that $\langle e_i, e^*_j \rangle = \delta_{ij}$, where $\{e_i\}$ is regarded as the collection of simple coroots and $\{e_i^*\}$ is regarded as the collection of fundamental weights. Given a cocharacter $\mu$, a character $\chi_\mu$ is uniquely determined by the following conditions
\begin{align*}
    \langle e_i, \chi_\mu \rangle = \langle \mu, e^*_i \rangle
\end{align*}
for each $i$. Similarly, given a character $\chi$, we define a cocharacter $\mu_{\chi}$ satisfying
\begin{align*}
    \langle \mu_\chi, e^*_i \rangle = \langle e_i, \chi \rangle
\end{align*}
for each $i$. Clearly,
\begin{align*}
    \langle \mu, \chi \rangle = \langle \mu_\chi, \chi_\mu \rangle.
\end{align*}
Moreover, a cocharacter $\mu$ determines a parabolic subgroup $P_\mu$ (see \S\ref{subsect_equiv_stab_local}). 

\begin{lem}\label{lem_char_cochar}
Given a cocharacter $\mu$, the character $\chi_\mu$ is a dominant character of $P_\mu$. On the other hand, given a character $\chi$, if it is a dominant character of some parabolic subgroup $P$, then $P \subseteq P_{\mu_\chi}$.
\end{lem}

\begin{proof}
Let $r$ be a root. By definition of $P_\mu$, we have $\langle \mu, r \rangle \geq 0$ for any root $r \in \mathcal{R}_{\mu}$. Suppose that $\chi_\mu$ is not dominant, then there exists a simple root $r_0 \in \mathcal{R}_\mu$ such that the coefficient of the corresponding fundamental weight $\lambda_{r_0}$ in $\chi_\mu$ is negative. This implies that $\langle \mu, r_0 \rangle <0$, which is a contradiction. The second statement is proved in a similar way.
\end{proof}

Given a collection of weights $\boldsymbol\theta$, let $d$ be the least common multiple of the denominators of weights in $\boldsymbol\theta$. Thus, $d \theta_x$ is a cocharacter for $x \in \boldsymbol{D}$. We define a character $\chi_{\boldsymbol{\theta}}: G_{\boldsymbol\theta} \rightarrow \mathbb{G}_m$ as a product
\begin{align*}
    \chi_{\boldsymbol{\theta}} : = \chi_{v_0} \cdot \prod_{x \in \boldsymbol{D}} \chi_{v_x},
\end{align*}
where $\chi_{v_0}$ is the trivial character and $\chi_{v_x} = \chi_{ - d \theta_x}$ for each $x \in \boldsymbol{D}$, which is a character of $P_{\theta_x}$ (and thus $L_{\theta_x}$) by Lemma \ref{lem_char_cochar} .

\begin{prop}\label{prop_stab_local_sys_and_quiv_rep_G}
Let $\phi = (\phi_a)_{a \in \mathscr{Q}^{\boldsymbol{D}}_1} \in \mathscr{R}(\mathscr{Q}^{\boldsymbol{D}},\mathscr{I}^{\boldsymbol{D}},\boldsymbol\theta)$ be a point. Denote by $\rho$ the corresponding $\boldsymbol\theta$-filtered $G$-local system. Then, $\rho$ is $R$-semistable (resp. $R$-stable) of degree zero if and only if the point $\phi$ is $\chi_{\boldsymbol\theta}$-semistable (resp. $\chi_{\boldsymbol\theta}$-stable).
\end{prop}

\begin{proof}
Given a cocharacter $\mu = (\mu_v)_{v \in \mathscr{Q}^{\boldsymbol{D}}_0}$ of $G_{\boldsymbol{\theta}}$, if the limit $\lim_{t\rightarrow 0} \mu(t) \cdot \phi$ exists, then for each $x \in \boldsymbol{D}$, there exists $g_x \in G$ such that $\phi_{c_{x,1}} g^{-1}_x \in L_{\theta_x}$, $g_x \phi_{c_{x,2}} \in P_{-\theta_x}$ and
\begin{align*}
    \mu_{v_x}(t) = g_x \mu_{v_0}(t) g_x^{-1}.
\end{align*}
The argument is exactly the same as Lemma \ref{lem_cochar_equal}. Denote by $P_\mu$ the parabolic subgroup determined by $P_{\mu_{v_0}}$.

First, we assume that the point $\phi$ is $\chi_{\boldsymbol\theta}$-semistable. By Proposition \ref{prop_king2.5}, the condition that $\chi_{\boldsymbol\theta}(\Delta) = \{1\}$ implies that $\rho$ is of degree zero, where $\Delta$ is the kernel of the $G_{\boldsymbol\theta}$-action. Then, we choose a parabolic subgroup $P$ compatible with $\rho$ and pick an arbitrary anti-dominant character $\chi$ of $P$. We have
\begin{align*}
( \boldsymbol\theta, \chi ) = \frac{1}{d} \sum_{x \in \boldsymbol{D}} ( d \theta_x, \chi ) =  \frac{1}{d} \sum_{x \in \boldsymbol{D}} ( -d \theta_x, -\chi ) = \frac{1}{d} \sum_{x \in \boldsymbol{D}} \langle -d \theta_x, -g_x \chi g^{-1}_x \rangle  
= \frac{1}{d} \sum_{x \in \boldsymbol{D}} \langle \mu_{ - g_x \chi g^{-1}_x} , \chi_{- d \theta_x} \rangle.
\end{align*}
The cocharacters $\mu_{- g_x \chi g^{-1}_x}$ and the element $\phi$ give a cocharacter $\mu : \mathbb{G}_m \rightarrow G_{\boldsymbol\theta}$, where 
\begin{align*}
    \mu = (\mu_{v})_{v \in \mathscr{Q}^{\boldsymbol{D}}_0},  \ \ \mu_{v_x} := \mu_{- g_x \chi g^{-1}_x}, \ \ \mu_{v_0} := g^{-1}_x \mu_{v_x} g_x.
\end{align*}
Therefore, we have
\begin{align*}
    ( \boldsymbol\theta, \chi ) = \frac{1}{d} \sum_{x \in \boldsymbol{D}} \langle \mu_{ - g_x \chi g^{-1}_x} , \chi_{- d \theta_x} \rangle = \frac{1}{d} \langle \mu, \chi_{\boldsymbol\theta}  \rangle.
\end{align*}
By Lemma \ref{lem_char_cochar}, we know that $P_\mu \supseteq P$, and then the limit $\lim_{t \rightarrow 0} \mu(t) \cdot \phi$ exists by the compatibility of $\rho$. By Proposition \ref{prop_king2.5}, we know that $\langle \mu, \chi_{\boldsymbol\theta} \rangle \geq 0$, and then, $( \boldsymbol\theta, \chi ) \geq 0$. Thus, $\rho$ is $R$-semistable.

Now we suppose that $\rho$ is $R$-semistable of degree zero. We take a cocharacter $\mu: \mathbb{G}_m \rightarrow G_{\boldsymbol\theta}$ such that the limit $\lim_{t \rightarrow 0} \mu(t) \cdot \phi$ exists. By Proposition \ref{prop_para_and_orbit_G}, we know
\begin{align*}
    \rho(a_i) = \phi_{a_i}, \ \ \rho(b_j) = \phi_{b_j}, \ \ \rho(c_x) = \phi_{c_{x,2}} \phi_{c_{x,1}},
\end{align*}
and limits
\begin{align*}
    \lim_{t \rightarrow 0} \mu_{v_0}(t) \rho(a_i) \mu_{v_0}(t)^{-1}, \ \ \lim_{t \rightarrow 0} \mu_{v_0}(t) \rho(b_j) \mu_{v_0}(t)^{-1}, \ \
    \lim_{t \rightarrow 0} \mu_{v_0}(t) \rho(c_x) \mu_{v_0}(t)^{-1}
\end{align*}
exist. Therefore, $\rho$ is compatible with $P_\mu$. Then,
\begin{align*}
\langle \mu, \chi_{\boldsymbol\theta} \rangle & = \sum_{x \in \boldsymbol{D}} \langle \mu_{v_x} , \chi_{- d \theta_x} \rangle = - d \sum_{x \in \boldsymbol{D}} \langle \theta_x , \chi_{\mu_{v_x}} \rangle  = - d \sum_{x \in \boldsymbol{D}} \langle g^{-1}_x \theta_x g_x, g^{-1}_x \chi_{\mu_{v_x}} g_x \rangle \\ 
& = - d \sum_{x \in \boldsymbol{D}} \langle g^{-1}_x \theta_x g_x, \chi_{\mu_{v_0}} \rangle = - d \sum_{x \in \boldsymbol{D}} ( \theta_x , \chi_{\mu_{v_0}} ) = -d ( \boldsymbol\theta , \chi_{\mu_{v_0}} ) .
\end{align*}
Since $\rho$ is $R$-semistable and $\chi_{\mu_{v_0}}$ is a dominant character of $P_\mu$, we have $( \boldsymbol\theta , \chi_{\mu_{v_0}} ) \leq 0$. Thus, $\langle \mu, \chi_{\boldsymbol\theta} \rangle \geq 0$. This finishes the proof for the semistable case. The argument for the stable case is similar.
\end{proof}

\subsection{Moduli Space of Filtered G-local Systems}

\begin{defn}
Let $\rho$ be a $\boldsymbol\theta$-filtered $G$-local system. A parabolic subgroup $P$ is said to be \emph{admissible} with $\rho$ if $P$ is compatible with $\rho$ and for any character $\chi: P \rightarrow \mathbb{G}_m$ trivial on the center, we have $\deg^{\rm loc} \rho(P,\chi) = 0$.
\end{defn}

Let $P$ be a parabolic subgroup with Levi decomposition $P=LU$. Given a $G$-local system $\rho$, if $P$ is compatible with $\rho$, the restriction of $\rho$ to the Levi subgroup $L$ is also well-defined and we obtain an $L$-local system. Taking the natural inclusion $L \rightarrow G$, the $L$-local system gives a $G$-local system and denote it by $\rho_L$. We give the definition of $S$-equivalence of filtered $G$-local systems analogous to \cite[Definition 3.6]{Ram96a}.

\begin{defn}
Two $R$-semistable $\boldsymbol\theta$-filtered $G$-local systems $\rho$ and $\rho'$ are \emph{$S$-equivalent} if there exist parabolic subgroups $P$ and $P'$ admissible with $\rho$ and $\rho'$ respectively such that the corresponding $G$-local systems $\rho_L$ and $\rho'_{L'}$ are isomorphic.
\end{defn}

\begin{lem}\label{lem_S_equiv_and_GIT_equiv_G}
Let $\phi = (\phi_a)_{a \in \mathscr{Q}^{\boldsymbol{D}}_1}$ and $\phi' = (\phi'_a)_{a \in \mathscr{Q}^{\boldsymbol{D}}_1}$ be two $\chi_{\boldsymbol\theta}$-semistable points in $\mathscr{R}(\mathscr{Q}^{\boldsymbol{D}},\mathscr{I}^{\boldsymbol{D}},\boldsymbol\theta)$. Let $\rho$ and $\rho'$ be the corresponding $R$-semistable $\boldsymbol\theta$-filtered $G$-local systems respectively. The filtered local systems $\rho$ and $\rho'$ are $S$-equivalent if and only if $\phi \sim \phi'$, i.e. they are GIT equivalent.
\end{lem}

\begin{proof}
This lemma is an application of Proposition \ref{prop_king2.6}, and the proof is similar to that of Lemma \ref{lem_S_equiv_and_GIT_equiv}. The analysis of degree follows directly from the proof and calculation in Proposition \ref{prop_stab_local_sys_and_quiv_rep_G}. The way we find an appropriate element in $G_{\boldsymbol\theta}$ and prove that $\lim_{t \rightarrow 0} \mu(t) \cdot \phi$ and $\lim_{t \rightarrow 0} \mu'(t) \cdot \phi'$ are in the same $G_{\boldsymbol\theta}$-orbit is exactly the same as Lemma \ref{lem_S_equiv_and_GIT_equiv}.
\end{proof}

We consider the moduli functor
\begin{align*}
\widetilde{\mathcal{M}}_{\rm B}(X_{\boldsymbol{D}},G,\boldsymbol\theta) : ({\rm Sch}/k)^{\rm op} \rightarrow {\rm Set}
\end{align*}
sending each $k$-scheme $S$ to the set of isomorphism classes of $S$-flat families of degree zero $R$-semistable $\boldsymbol\theta$-filtered $G$-local systems on $X_{\boldsymbol{D}}$. With a similar approach as Theorem \ref{thm_tame_Betti_lc}, we have the following results:

\begin{thm}\label{thm_tame_Betti_lc_G}
Define
\begin{align*}
    \mathcal{M}_{\rm B}(X_{\boldsymbol{D}},G,\boldsymbol\theta) := \mathscr{R}(\mathscr{Q}^{\boldsymbol{D}},\mathscr{I}^{\boldsymbol{D}},\boldsymbol{\theta}) /\!\!/ (G_{\boldsymbol{\theta}} , \chi_{\boldsymbol\theta})
\end{align*}
to be the GIT quotient with respect to $\chi_{\boldsymbol\theta}$. Then,
\begin{enumerate}
    \item $\mathcal{M}_{\rm B}(X_{\boldsymbol{D}},G,\boldsymbol\theta)$ is a quasi-projective variety;
    \item $\mathcal{M}_{\rm B}(X_{\boldsymbol{D}},G,\boldsymbol\theta)$ universally co-represents the moduli functor $\widetilde{\mathcal{M}}_{\rm B}(X_{\boldsymbol{D}},G,\boldsymbol\theta)$;
    \item the $k$-points are in one-to-one correspondence with $S$-equivalence classes of degree zero $R$-semistable $\boldsymbol\theta$-filtered $G$-local systems;
    \item there exists an open subset $\mathcal{M}^s_{\rm B}(X_{\boldsymbol{D}},G,\boldsymbol\theta)$, whose $k$-points correspond to isomorphism classes of degree zero $R$-stable $\boldsymbol\theta$-filtered $G$-local systems.
\end{enumerate}
\end{thm}

\section{Betti Moduli Spaces in tame NAHC for Principal Bundles}\label{sect_Betti_moduli}
In the tame NAHC for principal bundles on noncompact curves, we fix the data of monodromies/residues at punctures in order to give a precise description of the one-to-one correspondence. Keeping this background in mind, in this section, we give the construction of the Betti moduli spaces $\mathcal{M}_{\rm B}(X_{\boldsymbol{D}},G,\boldsymbol\gamma, M_{\boldsymbol\gamma})$ in the tame NAHC for principal bundles on noncompact curves (Theorem \ref{thm_Betti_moduli}), where $\boldsymbol\gamma$ is a collection of weights and $M_{\boldsymbol\gamma}$ is a collection of given monodromy data. We first briefly review the statement of the tame NAHC for principal bundles, and we refer the reader to \cite{HKSZ22} for more details. Let $\alpha,\beta,\gamma$ (resp. $\boldsymbol\alpha,\boldsymbol\beta,\boldsymbol\gamma$) be weights (resp. collections of weights). Let $\mathcal{G}_{\boldsymbol\alpha}$ and $\mathcal{G}_{\boldsymbol\beta}$ be parahoric group schemes (see \cite{KSZ21,HKSZ22} for instance). With respect to the relation in the following table,
\begin{table}[H]
\begin{tabular}{|c|c|c|c|}
\hline
& Dolbeault & de Rham & Betti \\
\hline
weights & $\alpha$ & $\beta=\alpha-(s_\alpha+\bar{s}_\alpha)$ & $\gamma=-(s_\alpha+\bar{s}_\alpha)$ \\
\hline
residues $\backslash$ monodromies & $\varphi_\alpha=s_\alpha+Y_\alpha$ & $\nabla_\beta$ & $M_\gamma$\\
\hline
\end{tabular}
\end{table}
\noindent where $\varphi_{\alpha} = s_{\alpha} + Y_{\alpha}$ is the Jordan decomposition with semisimple part $s_{\alpha}$ and nilpotent part $Y_{\alpha}$ with $(X_{\alpha}, H_{\alpha}, Y_{\alpha})$ a Kostant--Rallis triple, and
\begin{align*}
& \nabla_\beta =\alpha+(s_\alpha-\bar{s}_\alpha)-( H_\alpha + X_\alpha - Y_\alpha),\\
& M_\gamma = \exp\left( - 2 \pi i ( \alpha + s_\alpha - \bar{s}_\alpha) \right)\exp \left( 2 \pi i( H_\alpha + X_\alpha - Y_\alpha ) \right),
\end{align*}
we define the following categories:
\begin{itemize}
\item $\mathcal{C}^{s}_{\rm Dol}(X,\mathcal{G}_{\boldsymbol \alpha},\varphi_{\boldsymbol\alpha})$: the category of degree zero $R$-stable logahoric $\mathcal{G}_{\boldsymbol\alpha}$-Higgs torsors on $X$, and the Levi factors of the residues of the Higgs fields are $\varphi_{\boldsymbol\alpha}$ at punctures;

\item $\mathcal{C}^{s}_{\rm dR}(X,\mathcal{G}_{\boldsymbol\beta},\nabla_{\boldsymbol\beta})$: the category of degree zero $R$-stable logahoric $\mathcal{G}_{\boldsymbol\beta}$-connections on $X$ and the Levi factors of the residues of the connections are $\nabla_{\boldsymbol\beta}$ at punctures;

\item $\mathcal{C}^{s}_{\rm B}(X_{\boldsymbol D}, G,\boldsymbol\gamma, M_{\boldsymbol\gamma})$: the category of degree zero $R$-stable $\boldsymbol\gamma$-filtered $G$-local systems, of which the Levi factors of the monodromies around punctures are $M_{\boldsymbol\gamma}$.
\end{itemize}
The tame NAHC for principal bundles on noncompact curves is given as follows:
\begin{thm}[Theorem 1.1 in \cite{HKSZ22}]\label{thm_HKSZ1.1}
The following categories are equivalent
\begin{align*}
\mathcal{C}^{s}_{\rm Dol}(X,\mathcal{G}_{\boldsymbol\alpha},\varphi_{\boldsymbol\alpha}) \cong \mathcal{C}^{s}_{\rm dR}(X,\mathcal{G}_{\boldsymbol\beta},\nabla_{\boldsymbol\beta}) \cong \mathcal{C}^{s}_{\rm B}(X_{\boldsymbol D}, G,\boldsymbol\gamma, M_{\boldsymbol\gamma}).
\end{align*}
\end{thm}

From the statement of the theorem, it is clear that we have to fix the data of residues and monodromies in order to get a one-to-one correspondence among these three objects. This is also true in the classical case \cite[\S 5]{Sim90}, which is hidden behind the statement of the main theorem \cite[Theorem, p. 718]{Sim90}. Now we will construct the Betti moduli spaces $\mathcal{M}^{s}_{\rm B}(X_{\boldsymbol D}, G,\boldsymbol\gamma, M_{\boldsymbol\gamma})$ with given data $M_{\boldsymbol\gamma}$ of monodromies, which co-represents the moduli functor
\begin{align*}
\widetilde{\mathcal{M}}^{s}_{\rm B}(X_{\boldsymbol D}, G, \boldsymbol \gamma, M_{\boldsymbol \gamma}): ({\rm Sch}/\mathbb{C})^{\rm op} \rightarrow {\rm Sets},
\end{align*}
sending each $\mathbb{C}$-scheme $S$ to the set of isomorphism classes of $S$-flat families of degree zero $R$-stable $\boldsymbol\gamma$-filtered $G$-local systems on  $X_{\boldsymbol{D}}$ with $M_{\boldsymbol\gamma}$ being the Levi factors of monodromies. 

We first review the construction of the locally closed subset 
\begin{align*}
    \mathscr{R}(\mathscr{Q}^{\boldsymbol{D}},\mathscr{I}^{\boldsymbol{D}},\boldsymbol{\gamma}) \subseteq \mathscr{R}(\mathscr{Q}^{\boldsymbol{D}},\mathscr{I}^{\boldsymbol{D}},G)
\end{align*}
given in Construction \ref{cons_new_quiv_G}. We define a $(\prod_{x \in \boldsymbol{D}} G_x)$-action
\begin{align*}
    (\prod_{x \in \boldsymbol{D}} G_x) \times \mathscr{R}(\mathscr{Q}^{\boldsymbol{D}},G) & \rightarrow \mathscr{R}(\mathscr{Q}^{\boldsymbol{D}},G)
\end{align*}
as follows
\begin{align*}
    (g_x) \cdot (\phi_{a_i}, \phi_{b_j}, \phi_{c_{x,1}}, \phi_{c_{x,2}}) := (\phi_{a_i} , \phi_{b_j} ,  \phi_{c_{x,1}} g_x^{-1}, g_x 
 \phi_{c_{x,2}} ),
\end{align*}
where $G_x = G$ for every $x \in \boldsymbol{D}$. Denote by $\boldsymbol{P} = \{P_{\gamma_x}, x \in \boldsymbol{D}\}$ the collection of parabolic subgroups. Define a closed subvariety $\mathscr{R}(\mathscr{Q}^{\boldsymbol{D}}, \boldsymbol{P}) \subseteq \mathscr{R}(\mathscr{Q}^{\boldsymbol{D}},G)$, whose points $\phi= (\phi_a)_{a \in \mathscr{Q}^{\boldsymbol{D}}_1}$ satisfy the condition
\begin{align*}
    \phi_{c_{x,1}} \in L_{\gamma_x}, \quad \phi_{c_{x,2}} \in P_{-\gamma_x}
\end{align*}
for each $x \in \boldsymbol{D}$. Denote by $\mathscr{R}'(\mathscr{Q}^{\boldsymbol{D}},\boldsymbol{P}) \subseteq (\prod_{x \in \boldsymbol{D}} G_x) \times \mathscr{R}(\mathscr{Q}^{\boldsymbol{D}},G)$ the fiber product 
\begin{center}
\begin{tikzcd}
\mathscr{R}'(\mathscr{Q}^{\boldsymbol{D}},\boldsymbol{P}) \arrow[r, dotted] \arrow[d,dotted] & \mathscr{R}(\mathscr{Q}^{\boldsymbol{D}},\boldsymbol{P}) \arrow[d] \\
(\prod_{x \in \boldsymbol{D}} G_x) \times \mathscr{R}(\mathscr{Q}^{\boldsymbol{D}},G) \arrow[r] & \mathscr{R}(\mathscr{Q}^{\boldsymbol{D}},G)
\end{tikzcd}
\end{center}
Then, we define
\begin{align*}
\mathscr{R}(\mathscr{Q}^{\boldsymbol{D}},\boldsymbol{\gamma}) := \mathscr{R}'(\mathscr{Q}^{\boldsymbol{D}},\boldsymbol{P})|_{\mathscr{R}(\mathscr{Q}^{\boldsymbol{D}},G)}.
\end{align*}
Adding the relation $\mathscr{I}^{\boldsymbol{D}}$, we obtain a closed subvariety $\mathscr{R}(\mathscr{Q}^{\boldsymbol{D}},\mathscr{I}^{\boldsymbol{D}},\boldsymbol{\gamma}) \subseteq \mathscr{R}(\mathscr{Q}^{\boldsymbol{D}},\boldsymbol{\gamma})$, which is quasi-projective. 

Note that there is a natural map 
\begin{align*}
    \mathscr{R}(\mathscr{Q}^{\boldsymbol{D}},\boldsymbol{P})\rightarrow \prod_{x \in \boldsymbol{D}} P_{ - \gamma_x}, \quad 
    (\phi_a)_{a \in \mathscr{Q}^{\boldsymbol{D}}_1} \mapsto (\phi_{c_{x,2}} \phi_{c_{x,1}})_{x \in \boldsymbol{D}}.
\end{align*}
Thus we obtain the following composition
\begin{align*}
\mathscr{R}'(\mathscr{Q}^{\boldsymbol{D}},\boldsymbol{P}) \rightarrow \mathscr{R}(\mathscr{Q}^{\boldsymbol{D}},\boldsymbol{P}) \rightarrow \prod_{x \in \boldsymbol{D}} P_{ - \gamma_x} \rightarrow \prod_{x \in \boldsymbol{D}} L_{\gamma_x}.
\end{align*}
Fixing an element $M_{\boldsymbol{\gamma}} \in \prod_{x \in \boldsymbol{D}} L_{\gamma_x}$, denote by $\mathscr{R}'(\mathscr{Q}^{\boldsymbol{D}},\boldsymbol{P} , M_{\boldsymbol\gamma}) \subseteq \mathscr{R}'(\mathscr{Q}^{\boldsymbol{D}},\boldsymbol{P})$ the preimage, and then take the restriction
\begin{align*}
    \mathscr{R}(\mathscr{Q}^{\boldsymbol{D}},\boldsymbol{\gamma}, M_{\boldsymbol\gamma}) := \mathscr{R}'(\mathscr{Q}^{\boldsymbol{D}},\boldsymbol{P} , M_{\boldsymbol\gamma})|_{\mathscr{R}(\mathscr{Q}^{\boldsymbol{D}},G)},
\end{align*}
Adding the relation $\mathscr{I}^{\boldsymbol{D}}$, we get a quasi-projective variety $\mathscr{R}(\mathscr{Q}^{\boldsymbol{D}},\mathscr{I}^{\boldsymbol{D}},\boldsymbol{\gamma}, M_{\boldsymbol\gamma})$. Finally, we define the moduli space 
\begin{align*}
    \mathcal{M}_{\rm B}(X_{\boldsymbol D}, G,\boldsymbol\gamma, M_{\boldsymbol\gamma}) := \mathscr{R}(\mathscr{Q}^{\boldsymbol{D}},\mathscr{I}^{\boldsymbol{D}},\boldsymbol{\gamma}, M_{\boldsymbol\gamma})  /\!\!/ (G_{\boldsymbol\gamma}, \chi_{\boldsymbol\gamma}).
\end{align*}
Following the approach in \S\ref{sect_fgls}, this is the Betti moduli space expected.
\begin{thm}\label{thm_Betti_moduli}
There exists a quasi-projective variety $\mathcal{M}_{\rm B}(X_{\boldsymbol{D}},G,\boldsymbol\gamma, M_{\boldsymbol\gamma})$ as the moduli space of  degree zero $R$-semistable $\boldsymbol\gamma$-filtered $G$-local systems such that the Levi factors of monodromies around punctures are $M_{\boldsymbol\gamma}$. There exists an open subset $\mathcal{M}^s_{\rm B}(X_{\boldsymbol{D}},G,\boldsymbol\gamma, M_{\boldsymbol\gamma})$, whose points correspond to isomorphism classes of $R$-stable objects in $\mathcal{M}_{\rm B}(X_{\boldsymbol{D}},G,\boldsymbol\gamma, M_{\boldsymbol\gamma})$.
\end{thm}

As a direct result of the existence of the Betti moduli space, the categorical correspondence of tame NAHC also holds for the corresponding moduli spaces:
\begin{thm}[Theorem 1.2 in \cite{HKSZ22}]\label{thm_HKSZ1.2}
The following topological spaces are homeomorphic
\begin{align*}
\mathcal{M}^{\rm (top)}_{\rm Dol}(X, \mathcal{G}_{\boldsymbol \alpha}, \varphi_{\boldsymbol \alpha}) \cong  \mathcal{M}^{\rm (top)}_{\rm dR}(X, \mathcal{G}_{\boldsymbol \beta}, \nabla_{\boldsymbol \beta}) \cong \mathcal{M}^{\rm (top)}_{\rm B}(X_{\boldsymbol D},G, \boldsymbol \gamma, M_{\boldsymbol \gamma} ).
\end{align*}
\end{thm}

\begin{rem}
    Although the result of Theorem \ref{thm_HKSZ1.2} is given in \cite{HKSZ22}, the authors do not give a direct construction of the Betti moduli spaces in \cite{HKSZ22} and only prove that the de Rham moduli functor is equivalent to the Betti moduli functor.
\end{rem}

%\vspace{2mm}

%\textbf{Conflict of Interest}.
%On behalf of all authors, the corresponding author states that there is no conflict of interest.

%\vspace{2mm}

%\textbf{Data Availability Statement}.
%Since this is an article in the area of algebraic geometry regarding to the construction of Betti moduli spaces, data availability is not applicable to this article as no new data were created or analyzed in this study. 

\bibliographystyle{amsalpha}
\bibliography{ref_Betti}

\providecommand{\bysame}{\leavevmode\hbox to3em{\hrulefill}\thinspace}
\providecommand{\MR}{\relax\ifhmode\unskip\space\fi MR }
\providecommand{\MRhref}[2]{%
  \href{http://www.ams.org/mathscinet-getitem?mr=#1}{#2}
}
\providecommand{\href}[2]{#2}
\begin{thebibliography}{HLRV11}

\bibitem[AD20]{AD20}
S.~Amrutiya and U.~V. Dubey, \emph{Moduli of filtered quiver representations},
  Bull. Sci. Math. \textbf{164} (2020), 102899, 29.

\bibitem[Alf17]{Alf17}
D.~Alfaya, \emph{Moduli space of parabolic $\lambda$-modules over a curve},
  arXiv:1710.02080 (2017).

\bibitem[ASS06]{ASS06}
I.~Assem, D.~Simson, and A.~Skowro\'{n}ski, \emph{Elements of the
  representation theory of associative algebras. {V}ol. 1}, London Mathematical
  Society Student Texts, vol.~65, Cambridge University Press, Cambridge, 2006,
  Techniques of representation theory.

\bibitem[Bal23]{Bal23}
M.~Ballandras, \emph{Comet-shaped quiver varieties, {W}eyl group actions, and
  modified {K}ostka polynomials}, arXiv:2301.03434 (2023).

\bibitem[BBP17]{BBP17}
V.~Balaji, I.~Biswas, and Y.~Pandey, \emph{Connections on parahoric torsors
  over curves}, Publ. Res. Inst. Math. Sci. \textbf{53} (2017), no.~4,
  551--585.

\bibitem[BY96]{BY96}
H.~U. Boden and K.~Yokogawa, \emph{Moduli spaces of parabolic {H}iggs bundles
  and parabolic {$K(D)$} pairs over smooth curves. {I}}, Internat. J. Math.
  \textbf{7} (1996), no.~5, 573--598.

\bibitem[CBS06]{CBS06}
W.~Crawley-Boevey and P.~Shaw, \emph{Multiplicative preprojective algebras,
  middle convolution and the {D}eligne-{S}impson problem}, Adv. Math.
  \textbf{201} (2006), no.~1, 180--208.

\bibitem[HKSZ22]{HKSZ22}
P.~Huang, G.~Kydonakis, H.~Sun, and L.~Zhao, \emph{Tame parahoric nonabelian
  {H}odge correspondence on curves}, arXiv:2205.15475 (2022).

\bibitem[HL97]{HL97}
D.~Huybrechts and M.~Lehn, \emph{The geometry of moduli spaces of sheaves},
  Aspects of Mathematics, E31, Friedr. Vieweg \& Sohn, Braunschweig, 1997.

\bibitem[HLRV11]{HLRV11}
T.~Hausel, E.~Letellier, and F.~Rodriguez-Villegas, \emph{Arithmetic harmonic
  analysis on character and quiver varieties}, Duke Math. J. \textbf{160}
  (2011), no.~2, 323--400.

\bibitem[HS23]{HS22}
P.~Huang and H.~Sun, \emph{Meromorphic parahoric {H}iggs torsors and filtered
  {S}tokes {$G$}-local systems on curves}, Adv. Math. \textbf{429} (2023),
  Paper No. 109183, 38.

\bibitem[Ina13]{Ina13}
M.-A. Inaba, \emph{Moduli of parabolic connections on curves and the
  {R}iemann-{H}ilbert correspondence}, J. Algebraic Geom. \textbf{22} (2013),
  no.~3, 407--480.

\bibitem[IS08]{IS08}
J.~Iyer and C.~Simpson, \emph{The {C}hern character of a parabolic bundle, and
  a parabolic corollary of {R}eznikov's theorem}, Geometry and dynamics of
  groups and spaces, Progr. Math., vol. 265, Birkh\"auser, Basel, 2008,
  pp.~439--485.

\bibitem[Kin94]{King94}
A.~D. King, \emph{Moduli of representations of finite-dimensional algebras},
  Quart. J. Math. Oxford Ser. (2) \textbf{45} (1994), no.~180, 515--530.

\bibitem[Kon93]{Kon93}
H.~Konno, \emph{Construction of the moduli space of stable parabolic {H}iggs
  bundles on a {R}iemann surface}, J. Math. Soc. Japan \textbf{45} (1993),
  no.~2, 253--276.

\bibitem[KSZ24]{KSZ21}
G.~Kydonakis, H.~Sun, and L.~Zhao, \emph{Logahoric {H}iggs torsors for a
  complex reductive group}, Math. Ann. \textbf{388} (2024), no.~3, 3183--3228.

\bibitem[MFK94]{MFK94}
D.~Mumford, J.~Fogarty, and F.~Kirwan, \emph{Geometric invariant theory}, third
  ed., Ergebnisse der Mathematik und ihrer Grenzgebiete (2) [Results in
  Mathematics and Related Areas (2)], vol.~34, Springer-Verlag, Berlin, 1994.

\bibitem[MS80]{MS80}
V.~B. Mehta and C.~S. Seshadri, \emph{Moduli of vector bundles on curves with
  parabolic structures}, Math. Ann. \textbf{248} (1980), no.~3, 205--239.

\bibitem[Ram75]{Ram75}
A.~Ramanathan, \emph{Stable principal bundles on a compact {R}iemann surface},
  Math. Ann. \textbf{213} (1975), 129--152.

\bibitem[Ram96a]{Ram96a}
\bysame, \emph{Moduli for principal bundles over algebraic curves. {I}}, Proc.
  Indian Acad. Sci. Math. Sci. \textbf{106} (1996), no.~3, 301--328.

\bibitem[Ram96b]{Ram96b}
\bysame, \emph{Moduli for principal bundles over algebraic curves. {II}}, Proc.
  Indian Acad. Sci. Math. Sci. \textbf{106} (1996), no.~4, 421--449.

\bibitem[Rei08]{Rei08}
M.~Reineke, \emph{Moduli of representations of quivers}, Trends in
  representation theory of algebras and related topics, EMS Ser. Congr. Rep.,
  Eur. Math. Soc., Z\"{u}rich, 2008, pp.~589--637.

\bibitem[Ric88]{Rich88}
R.~W. Richardson, \emph{Conjugacy classes of {$n$}-tuples in {L}ie algebras and
  algebraic groups}, Duke Math. J. \textbf{57} (1988), no.~1, 1--35.

\bibitem[Sim90]{Sim90}
C.~T. Simpson, \emph{Harmonic bundles on noncompact curves}, J. Amer. Math.
  Soc. \textbf{3} (1990), no.~3, 713--770.

\bibitem[Sim92]{Sim92}
\bysame, \emph{Higgs bundles and local systems}, Inst. Hautes \'{E}tudes Sci.
  Publ. Math. (1992), no.~75, 5--95.

\bibitem[Sim94a]{Sim94a}
\bysame, \emph{Moduli of representations of the fundamental group of a smooth
  projective variety. {I}}, Inst. Hautes \'{E}tudes Sci. Publ. Math. (1994),
  no.~79, 47--129.

\bibitem[Sim94b]{Sim94b}
\bysame, \emph{Moduli of representations of the fundamental group of a smooth
  projective variety. {II}}, Inst. Hautes \'{E}tudes Sci. Publ. Math. (1994),
  no.~80, 5--79 (1995).

\bibitem[Sim97]{Sim97}
\bysame, \emph{The {H}odge filtration on nonabelian cohomology}, Algebraic
  geometry---{S}anta {C}ruz 1995, Proc. Sympos. Pure Math., vol.~62, Amer.
  Math. Soc., Providence, RI, 1997, pp.~217--281.

\bibitem[Sun20]{Sun20}
H.~Sun, \emph{Moduli spaces of $\lambda$-modules on projective
  {D}eligne-{M}umford stacks}, arXiv:2003.11674 (2020).

\bibitem[Yam08]{Yam08}
D.~Yamakawa, \emph{Geometry of multiplicative preprojective algebra}, Int.
  Math. Res. Pap. IMRP (2008), Art. ID rpn008, 77.

\bibitem[Yok93]{Yok93}
K.~Yokogawa, \emph{Compactification of moduli of parabolic sheaves and moduli
  of parabolic {H}iggs sheaves}, J. Math. Kyoto Univ. \textbf{33} (1993),
  no.~2, 451--504.

\end{thebibliography}

\bigskip
\noindent\small{\textsc{School of Mathematics, Nanjing University}\\
		 Nanjing 210093, China}\\
\emph{E-mail address}:  \texttt{pfhwangmath@gmail.com}

\bigskip
\noindent\small{\textsc{Department of Mathematics, South China University of Technology}\\
		381 Wushan Rd, Tianhe Qu, Guangzhou, Guangdong, China}\\
\emph{E-mail address}:  \texttt{hsun71275@scut.edu.cn}

\end{document}